\documentclass[a4paper,11pt]{article}

\usepackage{a4wide}
\usepackage{amsfonts,amssymb,amsmath,amsthm} 
\usepackage{srcltx,tabularx}
\usepackage[utf8]{inputenc} 
\usepackage[cyr]{aeguill}
\usepackage[alphabetic, nobysame]{amsrefs}
\usepackage{enumerate}
\usepackage[english]{babel}
\usepackage{authblk}
\usepackage[pagebackref=false, hidelinks]{hyperref}
\usepackage{cleveref}
\usepackage{graphicx}
\usepackage{pstricks,pst-plot,pst-node,pst-eps}
\usepackage[bottom]{footmisc}
\usepackage[all]{hypcap}

\newtheorem{theorem}{Theorem}[section]
\newtheorem*{theorem*}{Theorem}
\newtheorem{lemma}[theorem]{Lemma}

\newtheorem{proposition}[theorem]{Proposition}
\newtheorem{corollary}[theorem]{Corollary}
\newtheorem*{corollary*}{Corollary}
\newtheorem*{conjecture*}{Conjecture}
\newtheorem*{question*}{Question}
\newtheorem{claim}{Claim}

\newenvironment{claimproof}[1]{\par\noindent\textit{Proof of the claim:}\space#1}{\hfill $\blacksquare$}

\newtheorem{maintheorem}{Theorem}
\newtheorem{maincorollary}[maintheorem]{Corollary}

\theoremstyle{definition}
\newtheorem{definition}[theorem]{Definition}
\newtheorem*{definition*}{Definition}

\DeclareMathOperator{\Auttop}{\mathrm{Auttop}}
\DeclareMathOperator{\Aut}{\mathrm{Aut}}
\DeclareMathOperator{\Con}{\mathrm{Con}}
\DeclareMathOperator{\Par}{\mathrm{Par}}
\DeclareMathOperator{\Homeo}{\mathrm{Homeo}}
\DeclareMathOperator{\Stab}{\mathrm{Stab}}
\DeclareMathOperator{\Fix}{\mathrm{Fix}}
\DeclareMathOperator{\Cham}{\mathrm{Cham}}
\DeclareMathOperator{\proj}{\mathrm{proj}}
\DeclareMathOperator{\diam}{\mathrm{diam}}
\DeclareMathOperator{\Ker}{\mathrm{Ker}}
\DeclareMathOperator{\D}{\mathrm{D}}
\DeclareMathOperator{\Vertex}{\mathrm{Vert}}

\DeclareMathOperator{\id}{\mathrm{id}}
\DeclareMathOperator{\N}{\mathbf{N}}
\DeclareMathOperator{\Z}{\mathbf{Z}}
\DeclareMathOperator{\CG}{\mathcal{CG}}

\newcommand{\suchthat}{\;\ifnum\currentgrouptype=16 \middle\fi|\;}

\title{A topological characterization of the Moufang \\ property for compact polygons\footnotetext{2010 Mathematics Subject Classification: 20E42 51E24 (Primary), 20F65 22D05 (Secondary).}}
\author{Nicolas Radu\thanks{F.R.S.-FNRS Research Fellow.}}

\affil{UCLouvain, 1348 Louvain-la-Neuve, Belgium}
\date{July 15, 2016}

\begin{document}
 \maketitle
 
\begin{abstract}
We prove a purely topological characterization of the Moufang property for disconnected compact polygons in terms of convergence groups. As a consequence, we recover the fact that a locally finite thick affine building of rank~$3$ is a Bruhat--Tits building if and only if its automorphism group is strongly transitive. We also study automorphism groups of general compact polygons without any homogeneity assumption. A compactness criterion for sets of automorphisms is established, generalizing the theorem by Burns and Spatzier that the full automorphism group, endowed with the compact-open topology, is a locally compact group.
\end{abstract}

\setcounter{tocdepth}{2}

\tableofcontents

\section{Introduction}

Compact buildings form a natural generalization of compact projective planes. They were introduced by Burns and Spatzier in~\cite{Burns}, and used in their proof of the Rank Rigidity Theorem for Riemannian manifolds of non-positive sectional curvature, see \cite{Burns2}. A prominent example of a compact building is provided by the spherical building associated with a semisimple algebraic group over a non-discrete locally compact field. Since the origin, the problem of characterizing those buildings associated with semisimple algebraic groups among all compact buildings has attracted much attention; in some sense, this is the topic of the encyclopaedic book~\cite{Salzmann}, mainly devoted to the study of \emph{connected} compact projective planes. The relevance of that problem is also well illustrated by the work of Burns and Spatzier: a key result from \cite{Burns} that is used in the proof of Rank Rigidity in \cite{Burns2} asserts that a thick irreducible \emph{connected} compact spherical building of rank~$\geq 2$ is the spherical building associated with a simple Lie group if and only if its automorphism group is strongly transitive (see Subsection~2.2 for the definition of strong transitivity). This result has been improved by Grundh\"ofer, Knarr and Kramer who showed in \cite{GKK} and \cite{GKK2} that the same classification is true for connected buildings admitting a chamber transitive automorphism group. The starting point of the present article is the question whether such a statement could hold beyond the connected case. A precise formulation can be stated as~follows. 

\begin{conjecture*}
Let $\Delta$ be an infinite thick irreducible compact spherical building of rank at least~$2$. Then $\Delta$ is the spherical building associated with a semisimple algebraic group over a non-discrete locally compact field if and only if $\Delta$ is strongly transitive.
\end{conjecture*}

Here again, one could even replace strong transitivity by chamber transitivity in this conjecture. Historically, a similar problem had been suggested by J.~Tits in the 1970's for finite buildings. More precisely, Tits conjectured that a finite thick irreducible building of rank~$\geq 2$ should be the building associated with a finite simple group of Lie type if and only if the building has a strongly transitive automorphism group, see \cite{Tits}*{Conjecture~11.5.1}. Tits' hope was that this could be helpful to the  classification of the finite simple groups, which was an ongoing project at the time. Ironically, the latter classification was achieved first, and then used by Buekenhout and Van Maldeghem in their proof that Tits' conjecture is indeed accurate, see \cite{Buekenhout}.

As mentioned above, the conjecture is certainly true for connected buildings by the work of Burns and Spatzier; we may therefore assume that $\Delta$ is totally disconnected, since any disconnected compact building is so. Moreover, Grundh\"ofer, Kramer, Van Maldeghem and Weiss have proved that, under the same hypotheses as the conjecture, the building $\Delta$ is the spherical building associated with a semisimple algebraic group over a non-discrete locally compact field if and only if  $\Delta$ is Moufang: this follows from~\cite{Grundhofer}*{Theorem~1.1}. As is well known, Tits proved that every thick irreducible spherical building of rank~$\geq 3$ is Moufang (see \cite{Tits}*{Addenda}), so that the conjecture above can be reduced to the following one, which appears as Question~1.4 in \cite{Grundhofer}. 

\begin{conjecture*}[Reformulation]
Let $\Delta$ be an infinite thick irreducible compact totally disconnected spherical building of rank~$2$. If $\Delta$ is strongly transitive, then $\Delta$ is Moufang.
\end{conjecture*}

In the rest of this paper, we shall therefore focus on compact spherical buildings of rank~$2$, also called \textbf{compact polygons}. More precisely, a compact polygon of diameter $m$ is called a \textbf{compact $m$-gon}. It is irreducible if and only if $m \geq 3$.

Our first main result may be viewed as a first step towards the conjecture above.

\begin{maintheorem}\label{theorem:Characterization}
Let $\Delta$ be an infinite thick compact totally disconnected $m$-gon with $m \geq 3$ and let $G$ be a closed subgroup of the group $\Auttop(\Delta)$ of topological automorphisms  which is strongly transitive on $\Delta$. The following assertions are equivalent.
\begin{enumerate}[(i)]
\item $\Delta$ is Moufang.
\item For each panel $\pi$ of $\Delta$, the closure in $\Homeo(\Cham(\pi))$ of the group of projectivities $\Pi(\pi)$ acts as a convergence group on $\Cham(\pi)$.
\item For each panel $\pi$ of $\Delta$, the closure of the natural image of $\Stab_{G}(\pi)$ in $\Homeo(\Cham(\pi))$ acts as a convergence group on $\Cham(\pi)$.
\end{enumerate}
\end{maintheorem}

The notion of a convergence group is recalled in Subsection~\ref{subsection:convergence} while the group of projectivities $\Pi(\pi)$ is defined in Subsection~\ref{subsection:projectivities}. Conditions (i) and (ii) are both independent of the group $G$. In fact, Condition (ii) in Theorem~\ref{theorem:Characterization} implies in particular that the closure in $\Homeo(\Cham(\pi))$ of the group of projectivities $\Pi(\pi)$ is a locally compact group (see Lemma~\ref{lemma:lc} below). In the case of \emph{connected} compact polygons, the latter property in turn  is known to be equivalent to the Moufang condition, without assuming any extra homogeneity assumption on $\Delta$ a priori (see \cite{Lowen}*{Theorem~5.1} or \cite{Salzmann}*{Theorem~66.1}). If one believes that the connected case is representative of the general case, one should therefore not expect Condition~(ii) from Theorem~\ref{theorem:Characterization} to be satisfied by all infinite compact totally disconnected polygons.

We point out the following corollary of Theorem~\ref{theorem:Characterization}.

\begin{maincorollary}[Caprace--Monod]\label{corollary:affine}
Let $X$ be a locally finite thick irreducible affine building of rank~$3$. If $X$ is strongly transitive, then the building at infinity $X_\infty$ is Moufang. 
In particular $X$ is the Bruhat--Tits building associated with a semisimple algebraic group over a non-Archimedean local field.
\end{maincorollary}

The special case of Corollary~\ref{corollary:affine} where $X$ is of type $\tilde{A}_2$ was obtained by Van Maldeghem and Van Steen in~\cite{VM}*{Main Theorem}, while the general case has been obtained by Caprace and Monod in~\cite{CapraceMonod}*{Corollary~E} using CAT($0$) geometry. Our approach to Theorem~\ref{theorem:Characterization} is different, but was partly inspired by theirs: instead of the CAT($0$) Levi decomposition used in their work, we use  a purely algebraic counterpart of that result, due to Baumgartner and Willis~\cite{Baumgartner}, and valid in the realm of totally disconnected locally compact groups. 

\bigskip

The basic result that makes the structure theory of locally compact groups available in the study of compact buildings is the theorem, due to Burns--Spatzier and valid for all thick irreducible compact spherical building $\Delta$ of rank~$\geq 2$, asserting that the topological automorphism group $\Auttop(\Delta)$ is  a locally compact group, see \cite{Burns}*{Theorem~2.1}. Equivalently, there exists an identity neighbourhood in the group $\Auttop(\Delta)$ endowed with the compact-open topology that is compact. Our second main result is a compactness criterion for more general subsets of $\Auttop(\Delta)$ in the rank~$2$ case. Notice that we do not make any transitivity assumption on the automorphism group.

\begin{maintheorem}\label{theorem:Criterion}
Let $\Delta$ be a thick compact $m$-gon with $m \geq 3$ and $C, C'$ be opposite chambers of $\Delta$. Denote by $D_0$ and $D_1$ the two chambers adjacent to (but different from) $C$ in the apartment containing $C$ and $C'$ and by $v_0$ (resp. $v_1$) the common vertex of $C$ and $D_0$ (resp. $D_1$). Let also $E_0$ (resp. $E_1$) be a chamber having vertex $v_0$ (resp. $v_1$) but different from $C$ and $D_0$ (resp. $D_1$). Let finally $U$ and $U'$ be closed subsets of $\Cham\Delta$ such that every chamber of $U$ is opposite every chamber of $U'$. Then for all $\varepsilon > 0$, the (possibly empty) set
$$J_\varepsilon(C,C',U,U',E_0,E_1) := \left\{ \varphi \in \Auttop(\Delta) \suchthat
\begin{array}{l}
\varphi(C) \in U, \ \varphi(C') \in U', \\
\rho(\varphi(C), \varphi(B_{v_i}(D_i, \varepsilon))) \geq \varepsilon  \ \ \forall i \in \{0,1\},\\
\rho(\varphi(D_i), \varphi(B_{v_i}(C, \varepsilon))) \geq \varepsilon  \ \ \forall i \in \{0,1\},\\
\rho(\varphi(C), \varphi(E_i)) \geq \varepsilon  \ \ \forall i \in \{0,1\},\\
\rho(\varphi(D_i), \varphi(E_i)) \geq \varepsilon  \ \ \forall i \in \{0,1\}\\
\end{array}\
\right\}$$
is compact, where $\rho$ is some fixed metric associated to the topology on $\Cham\Delta$ and $B_{v}(C, r)$ denotes the open ball in $\Cham(v)$ centered at $C$ and of radius $r$ (with respect to the metric $\rho$).
\end{maintheorem}

\begin{figure}
\centering
\includegraphics{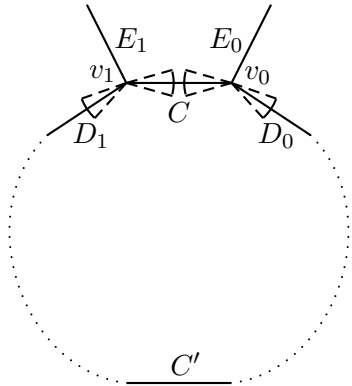}
\caption{Illustration of Theorem~\ref{theorem:Criterion}.}\label{picture:Criterion}
\end{figure}

Remark that the conclusion of Theorem~\ref{theorem:Criterion} is also valid for irreducible compact spherical buildings of higher rank, but the proof that we are aware of  in that case is very indirect: one deduces it from the fact that those buildings are Moufang, hence are buildings at infinity of Bruhat-Tits buildings, via the works of Tits~\cite{Tits}*{Addenda} and  Grundh\"ofer--Kramer--Van Maldeghem--Weiss~\cite{Grundhofer}*{Theorem~1.1}.

A particular case of this criterion is the following result, which can be seen as a topological version of a well-known theorem \cite{Tits}*{Theorem~4.1.1} of Tits. Indeed, the latter basically states that $L_0(C,C')$ (as defined below) is reduced to the identity in every thick (not necessarily topological) spherical building.

\begin{maincorollary}\label{corollary:Epsilon}
Let $\Delta$ be a thick compact $m$-gon with $m \geq 3$ and $C, C'$ be opposite chambers of $\Delta$. There exists $\varepsilon > 0$ such that the set
$$L_\varepsilon(C,C') := \left\{ \varphi \in \Auttop(\Delta) \mid \rho(X, \varphi (X)) \leq \varepsilon \ \ \forall X \in E_1(C) \cup \{C'\}
 \right\}$$
is compact, where $\rho$ is some fixed metric associated to the topology on $\Cham\Delta$ and $E_1(C)$ denotes the set of chambers adjacent to $C$.
\end{maincorollary}

We remark that the local compactness theorem of Burns and Spatzier in the rank~$2$ case can now be recovered as an immediate consequence of Corollary~\ref{corollary:Epsilon}. Actually, our proof of Theorem~\ref{theorem:Criterion} was inspired by their work; note however that  our approach is uniform in the sense that it does not require to distinguish the cases where the diameter of $\Delta$ is even or odd. In the particular case of compact projective planes (i.e. for $m = 3$), this result was also proved by Grundhöfer in a very short and elegant way, see~\cite{Grundhoferbis}*{Theorem~1}.

\begin{maincorollary}[Burns--Spatzier]\label{corollary:lc}
Let $\Delta$ be a thick compact $m$-gon with $m \geq 3$. Then $\Auttop(\Delta)$ is locally compact.
\end{maincorollary}

One could ask if the definition of the set $J_\varepsilon(C,C',U,U',E_0,E_1)$ whose compactness is ensured by Theorem~\ref{theorem:Criterion} is optimal: is there any natural bigger subset of $\Auttop(\Delta)$ whose compactness could be proved as well for all compact polygons? In that direction, we suggest the following question; we feel that the answer is likely to be positive.

\begin{question*}
Under the same hypotheses and with the same notation as in Theorem~\ref{theorem:Criterion}, is the following set compact?
$$\tilde{J}_\varepsilon(C,C',U,U',E_0,E_1) := \left\{ \varphi \in \Auttop(\Delta) \suchthat
\begin{array}{l}
\varphi(C) \in U, \ \varphi(C') \in U', \\
\rho(\varphi(C), \varphi(E_i)) \geq \varepsilon  \ \ \forall i \in \{0,1\},\\
\rho(\varphi(D_i), \varphi(E_i)) \geq \varepsilon  \ \ \forall i \in \{0,1\}\\
\end{array}\
 \right\}.$$
\end{question*}

Actually, this question is a natural generalization of the well-known conjecture that in any compact projective plane, the stabilizer of four points, none three of which  are collinear, is compact (see~\cite{Salzmann}*{Section~44}). It should also be emphasized that this question is strongly related to the conjecture mentioned at the beginning of this introduction. Indeed, in order to prove the conjecture, it only remains to show that Condition (ii) or (iii) in Theorem~\ref{theorem:Characterization} is satisfied by any strongly transitive compact polygon. A positive answer to the above question would ensure that a slightly weaker condition than (iii) holds for all compact polygons; in view of Theorem~\ref{theorem:Characterization}, this would be an important step toward the proof of the conjecture.

\bigskip

The proof of Theorem~\ref{theorem:Criterion} is given in Section~\ref{section:criterion}, while Theorem~\ref{theorem:Characterization} is proved in Section~\ref{section:characterization} (see Corollary~\ref{corollary:MoufangCG}, Proposition~\ref{proposition:proj-CG} and Theorem~\ref{theorem:Moufang}). Some of the preparatory results of Section~\ref{section:criterion} are also used in Section~\ref{section:characterization} to prove the implication (iii) $\Rightarrow$ (i) of Theorem~\ref{theorem:Characterization}.

\subsection*{Acknowledgement}

The results of this paper are the outcomes of my Master's thesis, which was supervised by Pierre-Emmanuel Caprace. I~am very grateful to him for suggesting the topic and for his support and guidance throughout the preparation of this work. I~also thank Linus Kramer for several comments on the history of compact polygons.

\section{Preliminaries}

In this section, we give the definition and first properties of compact spherical buildings and recall what the strong transitivity and the Moufang property exactly are.

\subsection{Compact spherical buildings}

Let $\Delta$ be a spherical building of rank~$k$ viewed as a simplicial complex, and denote by $\Delta_r$ the set of simplices of dimension~$r-1$ in $\Delta$ (for $r \in \{1, \ldots, k\}$). We fix an ordering of the $k$ types of vertices in $\Delta$. Then $\Delta$ is called a \textbf{compact spherical building} if the set $\Vertex\Delta := \Delta_1$ carries a compact topology such that the set $\Cham\Delta := \Delta_k$ is closed in $(\Vertex\Delta)^k$ (where a chamber of $\Delta$ is viewed as a $k$-tuple of vertices, ordered according to their type). For each $r \in \{1, \ldots, k\}$, $\Delta_r$ is given the induced topology from the product topology on $(\Vertex\Delta)^r$. We also say that $\Delta$ is connected or totally disconnected when $\Cham\Delta$ has the said property. When $\Delta$ is thick and has no factor of rank~$1$, one can actually show that the topology on $\Vertex\Delta$ is metrizable (see~\cite{Grundhofer}*{Proposition~6.14}). We will therefore always suppose that the topology comes from a metric. The \textbf{topological automorphism group} $\Auttop(\Delta)$ of $\Delta$ is the group of all automorphisms of $\Delta$ whose restriction to each $\Delta_r$ are homeomorphisms. The group $\Auttop(\Delta)$ is endowed with the topology induced from the compact-open topology of $C(\Cham\Delta,\Cham\Delta)$ (where $C(X,Y)$ denotes the set of continuous maps from $X$ to $Y$). Since $\Cham\Delta$ is compact and metrizable, it also actually is the topology of uniform convergence.

We can now give the first properties of compact buildings. The following results all come from article~\cite{Burns} of Burns and Spatzier and are almost direct consequences of the definition of compact buildings, except Lemma~\ref{lemma:non-isolated} whose proof is more technical.

\begin{proposition}\label{proposition:opposite}
Let $\Delta$ be a compact spherical building of rank~$k$. The set
$$\{(R,R') \in \Delta_r^2 \mid R \text{ is opposite } R'\}$$
is open in $\Delta_r^2$ for each $1 \leq r \leq k$.
\end{proposition}

\begin{proof}
See \cite{Burns}*{Proposition~1.9}.
\end{proof}

\begin{proposition}\label{proposition:projections}
Let $C$ be a chamber of a compact spherical building $\Delta$ and $R$ be a face of a chamber opposite $C$. If $R_n \to R$ and $C_n \to C$, then $\proj_{R_n}C_n \to \proj_R C$.
\end{proposition}

\begin{proof}
See \cite{Burns}*{Proposition~1.10}.
\end{proof}

\begin{corollary}\label{corollary:homeo}
Let $R$ and $R'$ be two opposite residues of a compact spherical building $\Delta$. The projection $\proj_R \mid_{R'} : \Cham R' \to \Cham R$ is a homeomorphism.
\end{corollary}

\begin{proof}
This follows from Proposition~\ref{proposition:projections}.
\end{proof}

In a polygon (i.e. a spherical building of rank~$2$), we will say that a gallery $(C_1, \ldots, C_n)$ \textbf{stammers} if $C_i = C_{i-1}$ for some $i \in \{2, \ldots, n\}$. The two following results are true in any usual (non-topological) polygon. Recall that an $m$-gon is a polygon of diameter $m$.

\begin{lemma}\label{lemma:minimal}
Let $\Delta$ be an $m$-gon. A gallery $(C_1, \ldots, C_n)$ in $\Delta$ with $n \leq m+1$ is minimal if and only if it does not stammer. In particular, there is no non-stammering closed gallery in $\Delta$ having less than $2m$ chambers.
\end{lemma}

\begin{proof}
This is a direct consequence of~\cite{Burns}*{Lemma~0.4}.
\end{proof}

\begin{corollary}
Let $\Delta$ be a polygon and $x, y$ be distinct non-opposite vertices of~$\Delta$. There is a unique minimal gallery with initial vertex $x$ and final vertex $y$, which we denote by $[x,y]$.
\end{corollary}

\begin{proof}
The existence of two minimal galleries from $x$ to $y$ would contradict Lemma~\ref{lemma:minimal}.
\end{proof}

When $\Delta$ is a polygon, we denote by $\D : \Vertex\Delta \times \Vertex\Delta \to \N$ the graph distance in the $1$-dimensional simplicial complex $\Delta$.

\begin{lemma}\label{lemma:gallery}
Let $\Delta$ be a compact polygon and $x, y$ be distinct non-opposite vertices of $\Delta$. If $x_n \to x$, $y_n \to y$ and $\D(x_n, y_n) = \D(x, y)$ for all $n \in \N$, then $[x_n, y_n] \to [x, y]$.
\end{lemma}

\begin{proof}
See~\cite{Burns}*{Lemma~1.12}.
\end{proof}

We say that two vertices $v$ and $v'$ in an $m$-gon are \textbf{almost opposite} if $\D(v,v') = m-1$ (note that $\D(v,v') = m$ if and only if $v$ and $v'$ are opposite). We then have a result similar to Proposition~\ref{proposition:opposite} for almost opposite vertices.

\begin{proposition}\label{proposition:almost-opposite}
Let $\Delta$ be a compact polygon. The set
$$\{(v,v') \in (\Vertex\Delta)^2 \mid v \text{ is almost opposite } v'\}$$
is open in $(\Vertex\Delta)^2$.
\end{proposition}

\begin{proof}
See~\cite{Burns}*{Lemma~1.13}.
\end{proof}

The following result will finally be helpful in the proofs of Theorems~\ref{theorem:Characterization} and~\ref{theorem:Criterion}.

\begin{lemma}\label{lemma:non-isolated}
Let $\Delta$ be an infinite thick compact $m$-gon with $m \geq 3$. For each $v \in \Vertex \Delta$, no chamber of $\Cham(v)$ is isolated in $\Cham(v)$. In particular, $\Cham(v)$ is infinite.
\end{lemma}

\begin{proof}
See~\cite{Burns}*{Lemma~1.14}.
\end{proof}

\subsection{Strong transitivity and the Moufang property}\label{subsection:transitivity}

Let $\Delta$ be a spherical building and $\Aut(\Delta)$ be the automorphism group of $\Delta$. A subgroup $G$ of $\Aut(\Delta)$ is said to be \textbf{strongly transitive} if $G$ is transitive on the set of all pairs $(A,C)$ where $A$ is an apartment of $\Delta$ and $C$ is a chamber of $A$. If $\Delta$ is a usual (non-topological) building, we say that $\Delta$ is \textbf{strongly transitive} when $\Aut(\Delta)$ is strongly transitive. If $\Delta$ is a compact building, this terminology is rather used when $\Auttop(\Delta)$ is strongly transitive.

The following fact is almost immediate but worth mentioning.

\begin{lemma}\label{lemma:strongtransitivity}
Let $\Delta$ be a spherical building and $\pi, \pi'$ be opposite panels of $\Delta$. Let also $G$ be a strongly transitive subgroup of $\Aut(\Delta)$. Then the stabilizer $\Stab_G(\pi, \pi')$ of $\pi$ and $\pi'$ in $G$ acts $2$-transitively on $\Cham(\pi)$. 
\end{lemma}

\begin{proof}
Let $C_1 \neq D_1$ and $C_2 \neq D_2$ be chambers in $\Cham(\pi)$. We want to find $g \in \Stab_G(\pi, \pi')$ such that $g(C_1) = C_2$ and $g(D_1) = D_2$. Let $A_1$ (resp. $A_2$) be the apartment of $\Delta$ containing $C_1$ (resp. $C_2$), $D_1$ (resp. $D_2$) and $\pi'$. By strong transitivity of $G$, there is $g \in G$ such that $g(A_1) = A_2$ and $g(C_1) = C_2$. Clearly, we also have $g(\pi) = \pi$, $g(\pi') = \pi'$ and $g(D_1) = D_2$.
\end{proof}

The Moufang property is another transitivity hypothesis on the automorphism group of a building. We say that a thick irreducible spherical building $\Delta$ is \textbf{Moufang} (or satisfies the \textbf{Moufang property}) if for each root $\alpha$ of $\Delta$, the \textbf{root group}
$$U_\alpha := \{g \in \Aut(\Delta) \mid g \text{ fixes every chamber having a panel in } \alpha \setminus \partial\alpha\}$$
acts transitively on the set of apartments of $\Delta$ containing $\alpha$. A natural wish would be to replace $\Aut(\Delta)$ with $\Auttop(\Delta)$ in the case where $\Delta$ is a compact building, but this is actually not necessary. Indeed, one can show that every element of a root group is automatically continuous, hence contained in $\Auttop(\Delta)$ (see~\cite{Grundhofer}*{Corollary~6.17}).

\section{A compactness criterion for subsets of \texorpdfstring{$\boldsymbol{\Auttop(\Delta)}$}{Auttop(Delta)}}\label{section:criterion}

This section is dedicated to the proof of Theorem~\ref{theorem:Criterion} which gives, in a compact polygon $\Delta$, a sufficient condition on subsets of $\Auttop(\Delta)$ for being compact.

\subsection[Relative compactness in \texorpdfstring{$\Auttop(\Delta)$}{Auttop(Delta)}]{Relative compactness in \texorpdfstring{$\boldsymbol{\Auttop(\Delta)}$}{Auttop(Delta)}}

To prove Theorem~\ref{theorem:Criterion}, we need some characterization of relative compactness in $\Auttop(\Delta)$. The tool that will help us is the Arzela-Ascoli theorem. For the reader's convenience, we recall its statement below. First recall that when $X$ is a topological space and $Y$ is a metric space, a subset $\mathcal{F}$ of $C(X,Y)$ is said to be \textbf{equicontinuous at $x_0 \in X$} if, for all $\varepsilon > 0$, there exists a neighbourhood $U$ of $x_0$ such that
$$f(U) \subseteq B(f(x_0), \varepsilon) \quad\text{for all } f \in \mathcal{F},$$
where $B(x, r)$ denotes the open ball of radius $r$ centered at $x$. The set $\mathcal{F}$ is then called \textbf{equicontinuous on $X$} if it is equicontinuous at each point of $X$.

\begin{theorem*}[Arzela-Ascoli]
Let $X$ be a topological space and $Y$ be a metric space. Let $\mathcal{F}$ be a subset of $C(X,Y)$. If $\mathcal{F}$ is equicontinuous on $X$ and if $\{f(x) \mid f \in \mathcal{F}\}$ is relatively compact in $Y$ for all $x \in X$, then $\mathcal{F}$ is relatively compact in $C(X,Y)$ equipped with the compact-open topology. The converse holds if $X$ is locally compact.
\end{theorem*}

\begin{proof}
See~\cite{Munkres}*{Theorem~47.1}.
\end{proof}

We can now already give some characterization of relative compactness in $\Auttop(\Delta)$.

\begin{lemma}\label{lemma:equicontinuous}
Let $\Delta$ be a compact spherical building. A set $S \subseteq \Auttop(\Delta)$ is relatively compact in $\Auttop(\Delta)$ if and only if $S$ and $S^{-1}$ are equicontinuous on $\Cham\Delta$.
\end{lemma}

\begin{proof}
First assume that $S$ is relatively compact in $\Auttop(\Delta)$. The set $S$ is thus also relatively compact in $C(\Cham \Delta, \Cham \Delta)$, hence $S$ is equicontinuous on $\Cham \Delta$ by Arzela-Ascoli. Since $\Auttop(\Delta)$ is a topological group, the inverse function of $\Auttop(\Delta)$ is a homeomorphism and $S^{-1}$ is in turn relatively compact and equicontinuous on $\Cham \Delta$.

Conversely, assume that $S$ and $S^{-1}$ are equicontinuous. We want to show that $S$ is relatively compact in $\Auttop(\Delta)$. We already know by Arzela-Ascoli that $S$ is relatively compact in $C(\Cham\Delta, \Cham\Delta)$. A way to conclude is therefore to prove that the closure of $S$ in $C(\Cham\Delta, \Cham\Delta)$ is contained in $\Auttop(\Delta)$. We thus consider a sequence $\varphi_n \to \varphi$ with $\{\varphi_n\} \subseteq \Auttop(\Delta)$ and $\varphi \in C(\Cham\Delta, \Cham\Delta)$ and prove that $\varphi \in \Auttop(\Delta)$. As a direct consequence of the fact that the adjacency relation is closed in a compact building, $\varphi$ is a (continuous) building morphism. To complete the proof, it suffices to find another such morphism $\psi$ so that $\varphi \psi = \psi \varphi = \id_\Delta$. Since $S^{-1}$ is equicontinuous and thus relatively compact in $C(\Cham\Delta, \Cham\Delta)$, it follows that some subsequence of $(\varphi_n^{-1})$ converges to an element of $C(\Cham\Delta, \Cham\Delta)$. Taking this limit for $\psi$, we obtain $\varphi \psi = \psi \varphi = \id_\Delta$.
\end{proof}

To prove that a given subset of $\Auttop(\Delta)$ is relatively compact, we will generally proceed by contradiction. In this context, the next proposition will be helpful. We introduce a convenient definition to state it.

\begin{definition}
Let $\Delta$ be a compact polygon. Two convergent sequences $a_n \to a$ and $b_n \to b$ in $\Vertex\Delta$ are said to be \textbf{collapsed by} a sequence $\{\varphi_n\} \subseteq \Auttop(\Delta)$ if $a \neq b$ and if $\varphi_n a_n \to x$ and $\varphi_n b_n \to x$ for some $x \in \Vertex\Delta$.
\end{definition}

\begin{proposition}\label{proposition:closer}
Let $\Delta$ be a compact polygon and $S \subseteq \Auttop(\Delta)$. If $S$ is not relatively compact in $\Auttop(\Delta)$, then there exist two sequences $a_n \to a$ and $b_n \to b$ in $\Vertex\Delta$ which are collapsed by some sequence $(\varphi_n)$ in $S$ or in $S^{-1}$.
\end{proposition}

\begin{proof}
Let $S$ be a subset of $\Auttop(\Delta)$ which is not relatively compact. By Lemma~\ref{lemma:equicontinuous}, $S$ or $S^{-1}$ is not equicontinuous on $\Cham\Delta$ and therefore also not equicontinuous on $\Vertex\Delta$. Suppose $S$ is not equicontinuous. Then there exist some $x \in \Vertex\Delta$ and $\varepsilon > 0$ such that for every neighbourhood $U$ of $x$, there exists $\psi \in S$ satisfying $\psi(U) \not \subseteq B(\psi x, \varepsilon)$. In other words, for every $n \in \N^*$ there exist $\psi_n \in S$ and $y_n \in B\left(x, \frac{1}{n}\right)$ such that $\psi_n y_n \not \in B(\psi_n x, \varepsilon)$. Clearly $y_n \to x$ and we can assume by passing to a subsequence that $\psi_n x \to a$ and $\psi_n y_n \to b$ for some $a \neq b$. Hence, the sequences $a_n = \psi_n x \to a$ and $b_n = \psi_n y_n \to b$ are collapsed by $\{\varphi_n = \psi_n^{-1}\} \subseteq S^{-1}$. The case where $S^{-1}$ is not equicontinuous is identical but gives two sequences which are collapsed by a sequence in $(S^{-1})^{-1} = S$.
\end{proof}

\subsection{Sequences collapsed by a sequence of automorphisms}

In this section, we consider a compact polygon $\Delta$. Proposition~\ref{proposition:closer} states that one can deduce from non-relative compactness of a subset $S$ of $\Auttop(\Delta)$ that there exist two sequences $a_n \to a$ and $b_n \to b$ of vertices of $\Delta$ collapsed by a sequence in $S$ or in $S^{-1}$. The goal of this subsection is to show that we can actually obtain additional constraints on these two sequences. The following result of Burns and Spatzier already goes in this direction by asserting that we can assume $\D(a_n, b_n) = \D(a, b) = 2$.

\begin{proposition}\label{proposition:centered}
Let $\Delta$ be a thick compact polygon. If there exist two sequences $a_n \to a$ and $b_n \to b$ in $\Vertex\Delta$ collapsed by a sequence $\{\varphi_n\} \subseteq \Auttop(\Delta)$, then there exist two sequences $a'_n \to a'$ and $b'_n \to b'$ collapsed by $(\varphi_n)$ and such that $\D(a'_n, b'_n) = \D(a', b') = 2$ for all $n \in \N$.
\end{proposition}

\begin{proof}
See \cite{Burns}*{Discussion between Assertions 2.3 and 2.4}.
\end{proof}

Now suppose that we are given a converging sequence of apartments $A_n \to A$, i.e. $2m$ converging sequences $v_i^{(n)} \to v_i$ ($i \in \{0, \ldots, 2m-1\}$) where $v_0^{(n)}, \ldots, v_{2m-1}^{(n)}$ (resp. $v_0, \ldots, v_{2m-1}$) are the vertices of an apartment $A_n$ (resp. $A$). In this context, we can generally even assume that the middle vertex $c_n$ of $a_n$ and $b_n$ (which are such that $\D(a_n, b_n) = 2$) is a vertex of $A_n$. For this reason we introduce the following definition.

\begin{definition}
Let $\Delta$ be a compact polygon. Two convergent sequences $a_n \to a$ and $b_n \to b$ in $\Vertex\Delta$ are said to be \textbf{centered at} $c_n \to c$ if $\D(a_n, c_n) = \D(c_n, b_n) = 1$ for all $n \in \N$. This also implies that $\D(a, c) = \D(c, b) = 1$.
\end{definition}

Before proving the announced result, we show the next lemma, which provides a way to go from a sequence of vertices to a sequence of opposite vertices.

\begin{lemma}\label{lemma:opposite}
Let $\Delta$ be a compact polygon and let $a_n \to a$ and $b_n \to b$ be two sequences in $\Vertex \Delta$ centered at $c_n \to c$ and collapsed by a sequence $\{\varphi_n\} \subseteq \Auttop(\Delta)$. Let also $c'_n \to c'$ be a sequence in $\Vertex \Delta$ such that $c'_n$ (resp. $c'$) is opposite $c_n$ (resp. $c$) for all $n \in \N$. Suppose that $\varphi_n c_n \to \tilde{c}$ and $\varphi_n c'_n \to \tilde{c}'$ for some opposite vertices $\tilde{c}$ and $\tilde{c}'$. Denote by $C_n$ (resp.~$D_n$) the chamber whose vertices are $c_n$ and $a_n$ (resp. $b_n$) and by $a'_n$ (resp. $b'_n$) the vertex of $\proj_{c'_n}(C_n)$ (resp. $\proj_{c'_n}(D_n)$) different from $c'_n$. Then the sequences $(a'_n)$ and $(b'_n)$ converge, are centered at $(c'_n)$ and are collapsed by $(\varphi_n)$.
\end{lemma}

\begin{proof}
\begin{figure}
\centering
\includegraphics{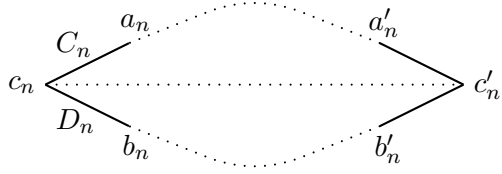}
\caption{Illustration of Lemma~\ref{lemma:opposite}.}\label{picture:opposite}
\end{figure}
By Proposition~\ref{proposition:projections}, the sequence $(a'_n)$ converges to $a'$, the vertex of $\proj_{c'}(C)$ different from $c'$ where $C$ is the chamber having vertices $c$ and $a$. The sequence $(b'_n)$ converges to $b'$ defined in the same way and $a' \neq b'$ since $a \neq b$. The fact that these two sequences are collapsed by $(\varphi_n)$ directly comes from the observation that the two sequences $\left(\varphi_n \proj_{c'_n}(C_n)\right) = \left(\proj_{\varphi_n c'_n}(\varphi_n C_n)\right)$ and $\left(\varphi_n \proj_{c'_n}(D_n)\right) = \left(\proj_{\varphi_n c'_n}(\varphi_n D_n)\right)$ have the same limit by Proposition~\ref{proposition:projections}.
\end{proof}

We draw attention to the fact that it is essential to have $\tilde{c}$ and $\tilde{c}'$ opposite each other to apply the previous lemma.

\begin{proposition}\label{proposition:centerok}
Let $\Delta$ be a thick compact polygon and $A_n \to A$ be a converging sequence of apartments of $\Delta$. If there exist two sequences $a_n \to a$ and $b_n \to b$ in $\Vertex\Delta$ collapsed by a sequence $\{\varphi_n\} \subseteq \Auttop(\Delta)$ and if $\varphi_n A_n \to \tilde{A}$ for some apartment~$\tilde{A}$ of $\Delta$, then there exist two sequences $a'_n \to a'$ and $b'_n \to b'$ centered at $c'_n \to c'$ and collapsed by some subsequence of $(\varphi_n)$, with $c'_n \in A_n$ for all $n \in \N$.
\end{proposition}

\begin{proof}
By Proposition~\ref{proposition:centered}, we can assume that the two sequences $a_n \to a$ and $b_n \to b$ are centered at some sequence of vertices $c_n \to c$. Let $d$ be a vertex of $A$ opposite $c$ (to find such a vertex, one can take a gallery from $c$ to any vertex of $A$ and extend this one with chambers of $A$ until an opposite vertex is reached) and let $d_n$ be the vertex of $A_n$ (for $n \in \N$) such that $d_n \to d$. Denote also by $\tilde{d}$ the limit of $(\varphi_n d_n)$. By Proposition~\ref{proposition:opposite}, $c_n$ is opposite $d_n$ for sufficiently large $n$. We can also assume by passing to a subsequence that $(\varphi_n c_n)$ converges to some $\tilde{c}$. If $\tilde{c}$ is opposite $\tilde{d}$, then we can immediately complete the proof using Lemma~\ref{lemma:opposite}.

\begin{figure}[b!]
\centering
\includegraphics{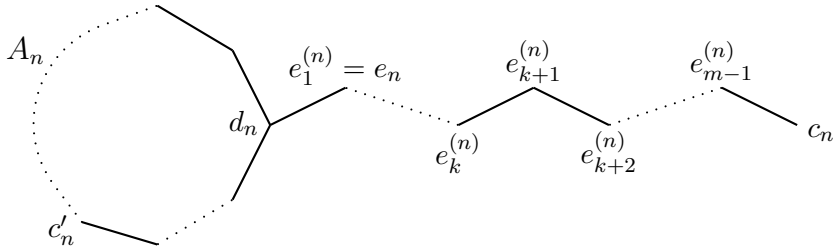}
\caption{Illustration of Proposition~\ref{proposition:centerok}.}\label{picture:centerok}
\end{figure}

If, on the contrary, $\tilde{c}$ is not opposite $\tilde{d}$, then we cannot proceed in this way. In this case, pick a sequence $e_n \to e$ with $e_n$ (resp. $e$) adjacent to $d_n$ (resp. $d$) but not contained in $A_n$ (resp. $A$). It is easy to show that such a sequence exists, using for example Proposition~\ref{proposition:projections}. Since $e$ is almost opposite $c$, $e_n$ is almost opposite $c_n$ for sufficiently large $n$ (Proposition~\ref{proposition:almost-opposite}). For these $n$, there is a gallery from $d_n$ to $c_n$ of length $m = \diam \Delta$ passing through $e_n$. Let $d_n = e_0^{(n)}, e_1^{(n)} = e_n, \ldots, e_{m-1}^{(n)}, e_m^{(n)} = c_n$ be the vertices of this gallery, as illustrated in Figure~\ref{picture:centerok}. After passing to subsequences, we can suppose that $e_i^{(n)} \to e_i$ and $\varphi_n e_i^{(n)} \to \tilde{e}_i$ for each $i \in \{0, \ldots, m\}$. In particular, $\tilde{e}_0 = \tilde{d}$ and $\tilde{e}_m = \tilde{c}$. Since $\tilde{c}$ is not opposite $\tilde{d}$, the gallery of vertices $\tilde{d} = \tilde{e}_0, \tilde{e}_1, \ldots, \tilde{e}_m = \tilde{c}$ must stammer. This means that there exists $k \in \{0, \ldots, m-2\}$ such that $\tilde{e}_k = \tilde{e}_{k+2}$ (if there is more than one such $k$, then we choose the smallest one). In other words, since $e_k \neq e_{k+2}$, the two sequences $e_k^{(n)} \to e_k$ and $e_{k+2}^{(n)} \to e_{k+2}$ are centered at $e_{k+1}^{(n)} \to e_{k+1}$ and collapsed by $(\varphi_n)$. Now take $c'$ a vertex of $A$ such that $\D(c',d) = m-k-1$ and $c'_n$ the vertex of $A_n$ (for $n \in \N$) such that $c'_n \to c'$. Thanks to this choice, $e_{k+1}$ is opposite $c'$ and $\tilde{e}_{k+1}$ is opposite the limit $\tilde{c}'$ of $(\varphi_n c'_n)$ (in the particular case where $\tilde{e}_1 \in \tilde{A}$, we choose $c'$ so that the minimal gallery from $\tilde{d}$ to $\tilde{c}'$ does not contain $\tilde{e}_1$). We can therefore apply Lemma~\ref{lemma:opposite} to the sequences $e_k^{(n)} \to e_k$ and $e_{k+2}^{(n)} \to e_{k+2}$ centered at $e_{k+1}^{(n)} \to e_{k+1}$ and the opposite sequence $c'_n \to c'$.
\end{proof}

The following result eventually shows that we can even suppose that one of the two sequences which are collapsed is contained in the sequence of apartments $A_n \to A$. The proof only works when the diameter $m$ of $\Delta$ is at least $3$, i.e. when $\Delta$ is irreducible.

\begin{proposition}\label{proposition:centerokbetter}
Let $\Delta$ be a thick compact $m$-gon with $m \geq 3$ and $A_n \to A$ be a converging sequence of apartments of $\Delta$. If there exist two sequences $a_n \to a$ and $b_n \to b$ in $\Vertex\Delta$ collapsed by a sequence $\{\varphi_n\} \subseteq \Auttop(\Delta)$ and if $\varphi_n A_n \to \tilde{A}$ for some apartment $\tilde{A}$ of $\Delta$, then there exist two sequences $a'_n \to a'$ and $b'_n \to b'$ centered at $c'_n \to c'$ and collapsed by some subsequence of $(\varphi_n)$, with $c'_n \in A_n$ and $a'_n \in A_n$ for all $n \in \N$.
\end{proposition}

\begin{proof}
By Proposition~\ref{proposition:centerok}, we can assume that the sequences $a_n \to a$ and $b_n \to b$ are centered at $v_n \to v$ with $v_n \in A_n$ (and thus $v \in A$). Let $\tilde{x}$ be the common limit of $(\varphi_n a_n)$ and $(\varphi_n b_n)$.

\begin{figure}[b!]
\centering
\includegraphics{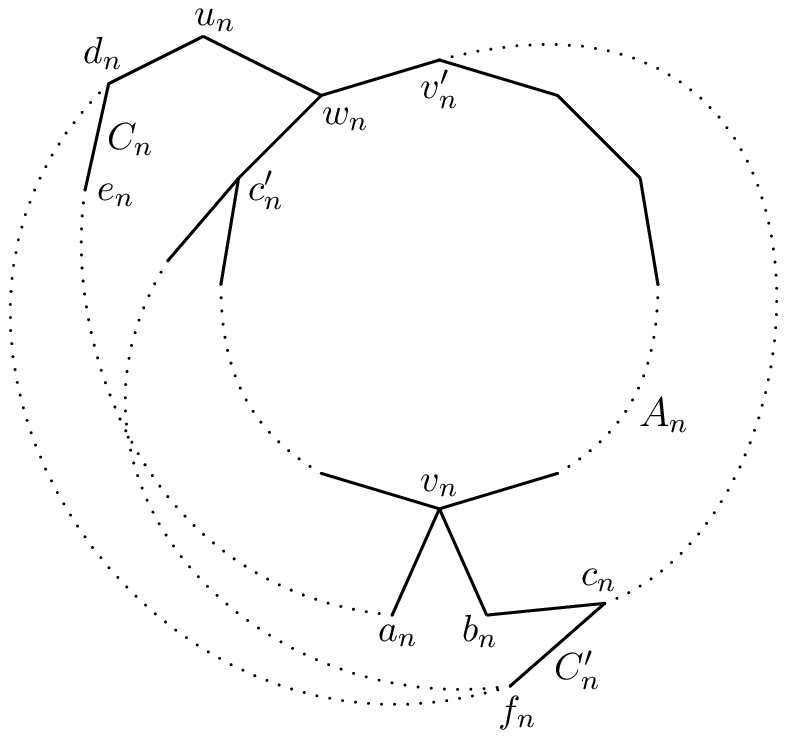}
\caption{Illustration of Proposition~\ref{proposition:centerokbetter}.}\label{picture:centerokbetter}
\end{figure}

If $\tilde{x}$ is a vertex of $\tilde{A}$, then consider $s_n \to s$ the sequence of vertices with $s_n \in A_n$ and $s \in A$ such that $\varphi_n s_n \to \tilde{x}$. Since $a \neq b$, we can assume without loss of generality that $a \neq s$. The sequences $a_n \to a$ and $s_n \to s$ are then collapsed by $(\varphi_n)$, which ends the proof.

Suppose now that $\tilde{x} \not \in \tilde{A}$. Let $v'_n$ (resp. $v'$) be the vertex of $A_n$ (resp. $A$) opposite $v_n$ (resp. $v$), $\tilde{v}'$ be the limit of $(\varphi_n v'_n)$ and $w_n \to w$ be a sequence of vertices where $w_n$ (resp.~$w$) is a vertex of $A_n$ (resp. $A$) and is adjacent to $v'_n$ (resp. $v'$). Consider also a sequence $u_n \to u$ with $u_n$ (resp. $u$) adjacent to $w_n$ (resp. $w$) but outside $A_n$ (resp. $A$). Passing to a subsequence if necessary, we can assume that $\varphi_n u_n \to \tilde{u}$ for some vertex $\tilde{u}$ (adjacent to the limit of $(\varphi_n w_n)$). Once again, if $\tilde{u} \in \tilde{A}$ then the proof is completed. We therefore assume that $\tilde{u} \not \in \tilde{A}$. For sufficiently large $n$, $a_n$ is almost opposite $u_n$ and $b_n$ is almost opposite $v'_n$ (see Proposition~\ref{proposition:almost-opposite}). We can thus draw $[a_n, u_n]$ and $[b_n, v'_n]$ as in Figure~\ref{picture:centerokbetter}. Thanks to Lemma~\ref{lemma:gallery}, $\varphi_n[a_n, u_n] \to [\tilde{x}, \tilde{u}]$ and $\varphi_n[b_n, v'_n] \to [\tilde{x}, \tilde{v}']$. The fact that $\tilde{u} \not \in \tilde{A}$ ensures that the two galleries $[\tilde{x}, \tilde{u}]$ and $[\tilde{x}, \tilde{v}']$ have no shared chambers. Let $c_n$ be the vertex of $[b_n, v'_n]$ adjacent to $b_n$, $d_n$ be the vertex of $[u_n, a_n]$ adjacent to $u_n$, $e_n$ be the vertex of $[d_n, a_n]$ adjacent to $d_n$, and $C_n$ be the chamber whose vertices are $d_n$ and $e_n$. The vertices $e_n$ and $c_n$ are opposite for large~$n$. Now, let $C'_n$ be the projection of $C_n$ on $c_n$ and $f_n$ be the vertex of $C'_n$ different from $c_n$. Since $d_n$ and $c_n$ are almost opposite as well as the limits of $(\varphi_n d_n)$ and $(\varphi_n c_n)$, Lemma~\ref{lemma:gallery} gives $\varphi_n f_n \to \tilde{x}$ and the two sequences $(b_n)$ and $(f_n)$ centered at $(c_n)$ are collapsed by $(\varphi_n)$.

Let finally $c'$ be a vertex of $A$ at distance $2$ from $v'$ such that $c'$ is opposite the limit of $(c_n)$ and let $c'_n$ be the vertex of $A_n$ (for $n \in \N$) such that $c'_n \to c'$. As the limits of $(\varphi_n c_n)$ and $(\varphi_n c'_n)$ are opposite, we can apply Lemma~\ref{lemma:opposite} to get two sequences of vertices centered at $c'_n \to c'$, collapsed by $(\varphi_n)$ and so that one of them is contained in $A_n \to A$.
\end{proof}

\subsection{Proof of the compactness criterion}

The following lemma is almost evident but will play an important role in the next results. As in the statement of Theorem~\ref{theorem:Criterion}, we write $B_v(C,r)$ for the open ball in $\Cham(v)$ centered at $C$ and of radius $r$.

\begin{lemma}\label{lemma:technical}
Let $\Delta$ be a compact polygon and $v^{(n)} \to v$ and $v'^{(n)} \to v'$ be two sequences in $\Vertex\Delta$ such that $v^{(n)}$ (resp. $v$) is opposite $v'^{(n)}$ (resp. $v'$) for all $n \in \N$. Let also $X^{(n)} \to X$ be a sequence in $\Cham\Delta$ such that $X^{(n)}$ (resp. $X$) has vertex $v^{(n)}$ (resp.~$v$) for all $n \in \N$. For each $\eta > 0$, there exists $\eta' > 0$ such that
$$B_{v'^{(n)}}(\proj_{v'^{(n)}}(X^{(n)}), \eta') \subseteq \proj_{v'^{(n)}}(B_{v^{(n)}}(X^{(n)}, \eta)) \ \text{ for all $n \in \N$.}$$
\end{lemma}

\begin{proof}
Fix $\eta > 0$ and suppose by contradiction that $\eta'$ does not exist. This means that we can find, after passage to a subsequence, a sequence $(Y^{(n)})$ with $Y^{(n)} \in B_{v'^{(n)}}(\proj_{v'^{(n)}}(X^{(n)}), \frac{1}{n})$ but $Y^{(n)} \not \in \proj_{v'^{(n)}}(B_{v^{(n)}}(X^{(n)}, \eta))$ for all $n \in \N^*$. By Proposition~\ref{proposition:projections}, $Y^{(n)} \to \proj_{v'}(X)$ and thus $\proj_{v^{(n)}}(Y^{(n)}) \to \proj_{v}(\proj_{v'}(X)) = X$. This contradicts the fact that $\proj_{v^{(n)}}(Y^{(n)}) \not \in B_{v^{(n)}}(X^{(n)}, \eta)$ for all $n \in \N^*$.
\end{proof}

The next definition will turn out to be convenient.

\begin{definition}
Let $X^{(n)} \to X$ and $Y^{(n)} \to Y$ be two sequences in $\Cham \Delta$ with $X^{(n)}$ (resp.~$X$) adjacent to $Y^{(n)}$ (resp. $Y$) in vertex $v^{(n)}$ (resp. $v$) for all $n \in \N$. Let $(\varphi_n)$ be a sequence in $\Auttop(\Delta)$ such that $\varphi_n Y^{(n)} \to \tilde{Y}$ for some $\tilde{Y} \in \Cham \Delta$. For $\eta > 0$, we say that $(X, Y)$ is \textbf{$\eta$-correct} (the sequences $X_n \to X$, $Y_n \to Y$ and $(\varphi_n)$ being clear from the context) if for each sequence $(T^{(n)})$ with $T^{(n)} \in B_{v^{(n)}}(X^{(n)}, \eta)$ the sequence $(\varphi_n T^{(n)})$ does not accumulate at $\tilde{Y}$. We say that $(X,Y)$ is \textbf{correct} if it is $\eta$-correct for some $\eta > 0$.
\end{definition}

The following lemma is the key result for the proof of Theorem~\ref{theorem:Criterion}. Indeed, it makes it possible to find some kind of uniform convergence from just a simple convergence. Figure~\ref{picture:key} should make the statement easier to digest.

\begin{lemma}\label{lemma:key}
Let $\Delta$ be a compact polygon and $A^{(n)} \to A$ be a converging sequence of apartments of $\Delta$. Let $v_0^{(n)}, \ldots, v_{2m-1}^{(n)}$ (resp. $v_0, \ldots, v_{2m-1}$) be the vertices of $A^{(n)}$ (resp.~$A$) such that $v_i^{(n)} \to v_i$ for $i \in \{0, \ldots, 2m-1\}$ and let $C_i^{(n)}$ (resp. $C_i$) be the chamber having vertices $v_i^{(n)}$ (resp. $v_i$) and $v_{i+1}^{(n)}$ (resp. $v_{i+1}$) for $i \in \{0, \ldots, 2m-1\}$, where $v_{2m}^{(n)} := v_0^{(n)}$ and $v_{2m} := v_0$.
Let also $b^{(n)} \to b \not \in A$ be a sequence in $\Vertex\Delta$ such that $b^{(n)}$ is adjacent to $v_0^{(n)}$ for all $n \in \N$. Fix $\eta > 0$. Then there exists $\eta' > 0$ such that if $(\varphi_n)$ is a sequence in $\Auttop(\Delta)$ satisfying
\begin{enumerate}[(a)]
\item $\varphi_n v_i^{(n)} \to \tilde{v}_i$ for each $i \in \{0, \ldots, 2m-1\}$ and the set of vertices $\tilde{v}_0, \ldots, \tilde{v}_{2m-1}$ form an apartment $\tilde{A}$ (whose chambers are denoted by $\tilde{C}_0, \ldots, \tilde{C}_{2m-1}$ with $\varphi_n C_i^{(n)} \to \tilde{C}_i$),
\item $(C_{m+1}, C_m)$ is $\eta$-correct (with respect to $(\varphi_n)$),
\item $\varphi_n b^{(n)} \to \tilde{v}_1$;
\end{enumerate}
then for each sequence $(R^{(n)})$ with $R^{(n)} \in B_{v_{m-1}^{(n)}}(C_{m-2}^{(n)}, \eta')$ for all $n \in \N$, $\varphi_n R^{(n)} \to \tilde{C}_{m-2}$.
\end{lemma}

\begin{proof}
\begin{figure}
\centering
\includegraphics{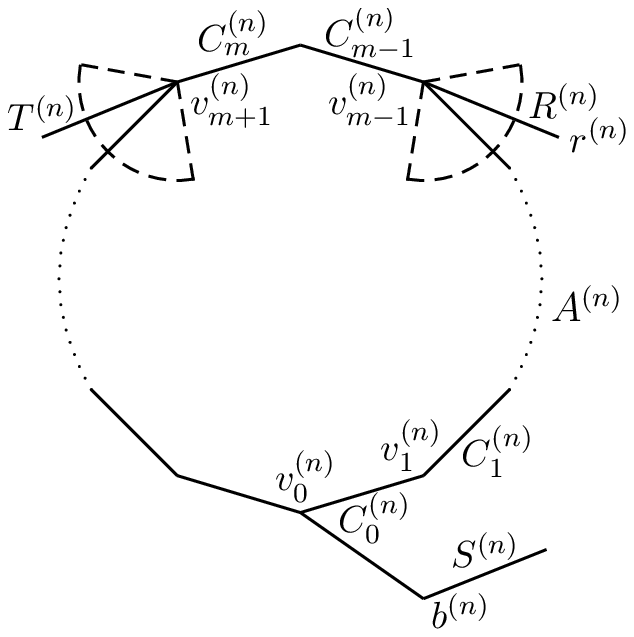}
\caption{Illustration of Lemma~\ref{lemma:key}.}\label{picture:key}
\end{figure}

Since $b \not\in A$, we can assume that $b^{(n)} \not \in A^{(n)}$ for all $n \in \N$.
We consider the balls $B_{v_{m+1}^{(n)}}(C_{m+1}^{(n)}, \eta)$ and apply Lemma~\ref{lemma:technical} twice: first with the sequences $v_{m+1}^{(n)} \to v_{m+1}$ and $b^{(n)} \to b$ and then with the sequences $b^{(n)} \to b$ and $v_{m-1}^{(n)} \to v_{m-1}$. This gives $\eta' > 0$ such that
$$B_{v_{m-1}^{(n)}}(C_{m-2}^{(n)}, \eta') \subseteq \proj_{v_{m-1}^{(n)}}(\proj_{b^{(n)}}(B_{v_{m+1}^{(n)}}(C_{m+1}^{(n)}, \eta))) \ \text{ for all $n \in \N$.} \quad (*)$$
We now prove that $\eta'$ satisfies the statement.

Consider a sequence $(R^{(n)})$ with $R^{(n)} \in B_{v_{m-1}^{(n)}}(C_{m-2}^{(n)}, \eta')$ and suppose for a contradiction that there exists a sequence $(\varphi_n)$ satisfying the conditions of the statement but such that $(\varphi_n R^{(n)})$ does not converge to $\tilde{C}_{m-2}$. Passing to a subsequence, we can assume that $\varphi_n R^{(n)} \to \tilde{R} \neq \tilde{C}_{m-2}$ and $R^{(n)} \to R$. Now take $S^{(n)} = \proj_{b^{(n)}}(R^{(n)})$ and $S = \proj_b(R)$. By Proposition~\ref{proposition:projections}, $S^{(n)} \to S$. Finally, define $T^{(n)} = \proj_{v_{m+1}^{(n)}}(S^{(n)})$ and $T = \proj_{v_{m+1}}(S)$. Once again, $T^{(n)} \to T$. By $(*)$, $T^{(n)} \in B_{v_{m+1}^{(n)}}(C_{m+1}^{(n)}, \eta)$ for all $n \in \N$. We now show that the convergence of $(\varphi_n R^{(n)})$ to $\tilde{R} \neq \tilde{C}_{m-2}$ necessarily implies $\varphi_n T^{(n)} \to \tilde{C}_m$, which is impossible since $(C_{m+1},C_m)$ is $\eta$-correct.

Denote by $r^{(n)}$ the vertex of $R^{(n)}$ different from $v_{m-1}^{(n)}$ and by $\tilde{r}$ the vertex of $\tilde{R}$ different from $\tilde{v}_{m-1}$. Since $b^{(n)}$ and $r^{(n)}$ are almost opposite (for large enough $n$) as well as $\tilde{v}_1$ and~$\tilde{r}$, $\varphi_n[b^{(n)}, r^{(n)}] \to [\tilde{v}_1, \tilde{r}]$ (Lemma~\ref{lemma:gallery}) and in particular $\varphi_n S^{(n)} \to \tilde{C}_1$. As $v_1$ is opposite $v_{m+1}$, we therefore get by Proposition~\ref{proposition:projections}:
\[
\varphi_n T^{(n)} = \varphi_n \left( \proj_{v_{m+1}^{(n)}}(S^{(n)}) \right) = \proj_{\varphi_nv_{m+1}^{(n)}}(\varphi_n S^{(n)}) \to \proj_{\tilde{v}_{m+1}}(\tilde{C}_1) = \tilde{C}_m. \qedhere
\]  
\end{proof}

We are now able to prove Theorem~\ref{theorem:Criterion}.

\begin{proof}[Proof of Theorem~\ref{theorem:Criterion}]
If $\Delta$ is finite, then $\Auttop(\Delta)$ and $J_\varepsilon := J_\varepsilon(C,C',U,U',E_0,E_1)$ are finite and the latter is thus compact. We will therefore assume that $\Delta$ is infinite. By Lemma~\ref{lemma:non-isolated}, this implies that $\Cham(v)$ is infinite for each vertex $v$ of $\Delta$.

Since $J_\varepsilon$ is closed in $\Auttop(\Delta)$, it is compact if and only if it is relatively compact in $\Auttop(\Delta)$. Suppose for a contradiction that it is not relatively compact. By Proposition~\ref{proposition:closer}, there exist two sequences $a^{(n)} \to a$ and $b^{(n)} \to b$ collapsed by a sequence $(\varphi_n)$ in $J_\varepsilon$ or in $J_\varepsilon^{-1}$. Observe both cases:

\begin{enumerate}[(1)]
\item If $\{\varphi_n\} \subseteq J_\varepsilon$, then we define $C^{(n)} = C$, $C'^{(n)} = C'$, $D_i^{(n)} = D_i$ and $E_i^{(n)} = E_i$ for all $n \in \N$ and $i \in \{0,1\}$. Let also $A^{(n)} = A$ be the apartment containing $C$ and $C'$ (for all $n \in \N$). Passing to a subsequence, we can assume that $\varphi_n C^{(n)} \to \tilde{C}$ and $\varphi_n C'^{(n)} \to \tilde{C}'$. By definition of $J_\varepsilon$, $\tilde{C}$ and $\tilde{C}'$ are opposite, which means that $\varphi_n A^{(n)} \to \tilde{A}$, the apartment containing $\tilde{C}$ and $\tilde{C}'$. The definition of $J_\varepsilon$ is also such that $(C, D_i)$ and $(D_i, C)$ are $\varepsilon$-correct for $i \in \{0, 1\}$. Finally, we can assume that $\varphi_n E_i^{(n)} \to \tilde{E}_i$ and $\tilde{E}_i \not \in \tilde{A}$ for $i \in \{0,1\}$.

\item If $\{\varphi_n\} \subseteq J_\varepsilon^{-1}$, then we define $F^{(n)} = \varphi_n^{-1} C$, $F'^{(n)} = \varphi_n^{-1} C'$, $G_i^{(n)} = \varphi_n^{-1} D_i$ and $H_i^{(n)} = \varphi_n^{-1} E_i$ for all $n \in \N$ and $i \in \{0,1\}$. Passing to a subsequence, we can assume that $F^{(n)} \to F$, $F'^{(n)} \to F'$, $G_i^{(n)} \to G_i$ and $H_i^{(n)} \to H_i$. Note that $F$ is opposite $F'$ since $\{\varphi_n^{-1}\} \subseteq J_\varepsilon$, which allows us to denote by $A^{(n)}$ (resp. $A$) the apartment containing $F^{(n)}$ (resp. $F$) and $F'^{(n)}$ (resp. $F'$) for all $n \in \N$. By definition of $J_\varepsilon$, $H_i \not \in A$. Moreover, $\varphi_n F^{(n)} = C \to C =: \tilde{F}$, $\varphi_n F'^{(n)} = C' \to C'=: \tilde{F'}$ and $\varphi_n H_i^{(n)} = E_i \to E_i =: \tilde{H}_i$. Hence, $\varphi_n A^{(n)} \to \tilde{A}$, the apartment containing $\tilde{F}$ and $\tilde{F}'$. It is easy to see that $(F, G_i)$ and $(G_i, F)$ are $\varepsilon$-correct (with respect to $(\varphi_n)$) and clearly $\tilde{H}_i \not \in \tilde{A}$ for $i \in \{0,1\}$.
\end{enumerate}

The situations in (1) and (2) are exactly the same, replacing letters $F$, $G$ and $H$ in (2) by letters $C$, $D$ and $E$ respectively. For this reason, we from now on assume that we are in case (1) and only use the objects $C^{(n)} \to C$, $C'^{(n)} \to C'$, $D_i^{(n)} \to D_i$ and $E_i^{(n)} \to E_i$ satisfying the conditions stated above. Let $v_0, \ldots, v_{2m-1}$ be the vertices of $A$ in the natural order so that $E_0$ has vertex $v_0$ and $E_1$ has vertex $v_1$. Once again, $v_0^{(n)}, \ldots, v_{2m-1}^{(n)}$ denotes the corresponding vertices in $A^{(n)}$ and $\tilde{v}_0, \ldots, \tilde{v}_{2m-1}$ the corresponding vertices in $\tilde{A}$.

\begin{claim}\label{claim:correct}
$(X,Y)$ is correct for any two adjacent chambers $X$ and $Y$ in $A$.
\end{claim}

\begin{claimproof}
\begin{figure}
\centering
\includegraphics{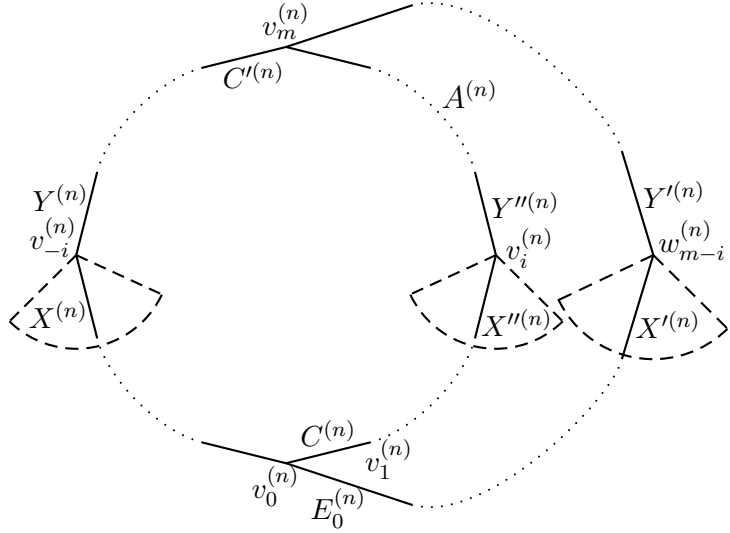}
\caption{Illustration of Claim~\ref{claim:correct}.}\label{picture:nonconvergence}
\end{figure}
In the proof of this claim, we say that $v_i$ is correct if $(X,Y)$ and $(Y,X)$ are correct where $X$ and $Y$ are the two chambers of $A$ having vertex $v_i$. We already know that $v_0$ and $v_1$ are correct, and we want to prove that every vertex of $A$ is correct.

We actually show that if $v_i$ is correct, then $v_{-i}$ and $v_{2-i}$ are also correct (the indices being considered modulo $2m$). To do so, draw the minimal gallery from $E_0$ (resp. $E_0^{(n)}$) to $v_m$ (resp. $v_m^{(n)}$) and denote by $w_1, \ldots, w_{m-1}$ (resp. $w_1^{(n)}, \ldots, w_{m-1}^{(n)}$) its vertices (see Figure~\ref{picture:nonconvergence}). By Lemma~\ref{lemma:gallery}, $w_i^{(n)} \to w_i$ for each $i \in \{1,\ldots, m-1\}$. Now assume that $v_i$ is correct for some $i \in \{1, \ldots, m-1\}$. To show that $v_{-i}$ is also correct, consider $X$ a chamber of $A$ having vertex $v_{-i}$ and $Y$ the other chamber of this apartment having vertex $v_{-i}$. Project these two chambers on $w_{m-i}$ by defining $X' = \proj_{w_{m-i}}(X)$ and $Y' = \proj_{w_{m-i}}(Y)$, and then on $v_i$ by defining $X'' = \proj_{v_i}(X')$ and $Y'' = \proj_{v_i}(Y')$. The chambers $X''$ and $Y''$ obviously are the chambers of $A$ having vertex $v_i$. Define $X^{(n)}$, $Y^{(n)}$, $X'^{(n)}$, $Y'^{(n)}$, $X''^{(n)}$ and $Y''^{(n)}$ in the natural way. Since $v_i$ is correct, there exists $\eta > 0$ such that $(X'',Y'')$ is $\eta$-correct. We can assume that $w_{m-i}^{(n)}$ is opposite $v_i^{(n)}$ and $v_{-i}^{(n)}$ for all $n \in \N$ since $E_0 \not \in A$. This allows us to go from $v_i$ to $v_{-i}$ via $w_{m-i}$ using Lemma~\ref{lemma:technical} twice. This gives $\eta' > 0$ such that
$$B_{v_{-i}^{(n)}}(X^{(n)}, \eta') \subseteq \proj_{v_{-i}^{(n)}}(\proj_{w_{m-i}^{(n)}}(B_{v_i^{(n)}}(X''^{(n)}, \eta))) \ \text{ for all $n \in \N$.} \quad (**)$$
We now show that $(X,Y)$ is $\eta'$-correct. Suppose for a contradiction that there exists $(T^{(n)})$ with $T^{(n)} \in B_{v_{-i}^{(n)}}(X^{(n)}, \eta')$ such that $(\varphi_n T^{(n)})$ accumulates at $\tilde{Y}$, the limit of $(\varphi_{n} Y^{(n)})$. Considering $T'^{(n)} = \proj_{w_{m-i}^{(n)}}(T^{(n)})$ and $T''^{(n)} = \proj_{v_i^{(n)}}(T'^{(n)})$, Proposition~\ref{proposition:projections} directly shows that $(\varphi_n T''^{(n)})$ accumulates at $\tilde{Y}''$, the limit of $(\varphi_{n} Y''^{(n)})$. This is a contradiction since $(X'',Y'')$ is $\eta$-correct and $T''^{(n)} \in  B_{v_i^{(n)}}(X''^{(n)}, \eta)$ for all $n \in \N$ by $(**)$. Note that we assumed $i \in \{1,\ldots,m-1\}$ but the reasoning is obviously the same for $i \in \{m+1, \ldots,2m-1\}$.

The same construction with $E_1$ instead of $E_0$ enables us to go from $v_i$ to $v_{2-i}$. We directly conclude that all vertices are correct since $v_0$ and $v_1$ are correct and $v_{i+2} = v_{2-(-i)}$ is correct as soon as $v_i$ is. The claim stands proven.
\end{claimproof}

\medskip

From now on, we denote by $C_i$ the chamber having vertices $v_i$ and $v_{i+1}$ (indices being taken modulo $2m$).

\begin{claim}\label{claim:b-c}
After possible passage to a subsequence, there exists for each $i \in \{0, \ldots, 2m-1\}$ a sequence $b_i^{(n)} \to b_i \not \in A$ with $b_i^{(n)}$ adjacent to $v_i^{(n)}$ for all $n \in \N$ and $\varphi_n b_i^{(n)} \to \tilde{v}_{i+1}$. Similarly, there exists a sequence $c_i^{(n)} \to c_i \not \in A$ with $c_i^{(n)}$ adjacent to $v_i^{(n)}$ and $\varphi_n c_i^{(n)} \to \tilde{v}_{i-1}$.
\end{claim}

\begin{claimproof}
\begin{figure}
\centering
\includegraphics{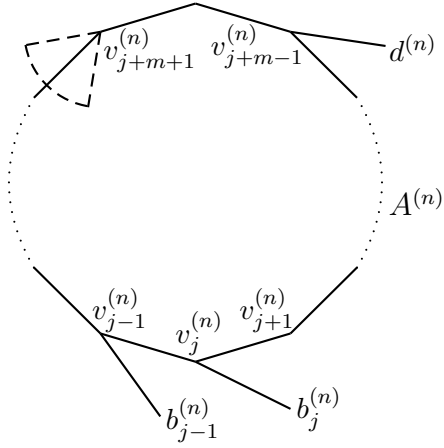}
\caption{Illustration of Claim~\ref{claim:b-c}.}\label{picture:A1}
\end{figure}
Recall that there are two sequences $a^{(n)} \to a$ and $b^{(n)} \to b$ collapsed by $(\varphi_n)$. All the hypotheses of Proposition~\ref{proposition:centerokbetter} being met, we can assume that $a^{(n)} \to a$ and $b^{(n)} \to b$ are centered at some $v_j^{(n)} \to v_j$ and that $a^{(n)} \in A^{(n)}$ for all $n \in \N$. Without loss of generality, we may assume that $a^{(n)} = v_{j+1}^{(n)}$. Clearly $b \neq v_{j+1}$, and $b \neq v_{j-1}$ since $(C_{j-1},C_j)$ is correct. The sequence $b_j^{(n)} \to b_j$ defined by $b_j^{(n)} = b^{(n)}$ satisfies the statement of the claim for $i=j$.

We now prove that one can construct $b_{j-1}^{(n)} \to b_{j-1}$ from $b_j^{(n)} \to b_j$. Since $(C_{j+m+1},C_{j+m})$ is correct, there is by Lemma~\ref{lemma:key} a sequence $d^{(n)} \to d \not \in A$ with $d^{(n)}$ adjacent to $v_{j+m-1}^{(n)}$ such that $\varphi_n d^{(n)} \to \tilde{v}_{j+m-2}$. After projection on $v_{j-1}^{(n)} \to v_{j-1}$ (see Lemma~\ref{lemma:opposite}), there is a sequence $b_{j-1}^{(n)} \to b_{j-1} \not \in A$ with $b_{j-1}^{(n)}$ adjacent to $v_{j-1}^{(n)}$ such that $\varphi_n b_{j-1}^{(n)} \to \tilde{v}_j$. By repeating this process, we get the sequences $b_i^{(n)} \to b_i$ satisfying the statement for each $i \in \{0, \ldots, 2m-1\}$. Using Lemma~\ref{lemma:opposite}, there obviously also is for each vertex $v_i$ a sequence $c_i^{(n)} \to c_i \not \in A$ with $c_i^{(n)}$ adjacent to $v_i^{(n)}$ such that $\varphi_n c_i^{(n)} \to \tilde{v}_{i-1}$.
\end{claimproof}

\medskip

Let us now focus on $\Cham(v_0)$. In the rest of this proof, we associate to each $X \in \Cham(v_0)$ a sequence $X^{(n)} \to X$ with $X^{(n)} \in \Cham(v_0^{(n)})$. Such a sequence exists, take for instance $X^{(n)} = \proj_{v_0^{(n)}}(\proj_{v_m}(X))$ ($v_m$ being opposite $v_0^{(n)}$ for sufficiently large $n$).

\begin{claim}\label{claim:chamv0}
For each $X \in \Cham(v_0)$, there exists $\eta_X > 0$ such that if $\varphi_{k(n)} X^{(k(n))} \to \tilde{X}$ for some subsequence $(\varphi_{k(n)})$ of $(\varphi_n)$ and some chamber $\tilde{X}$, then $\varphi_{k(n)} R^{(k(n))} \to \tilde{X}$ for each sequence $(R^{(n)})$ with $R^{(n)} \in B_{v_0^{(n)}}(X^{(n)}, \eta_X)$ for all $n \in \N$.
\end{claim}

\begin{claimproof}
Let $X \in \Cham(v_0)$ and first assume that $X \neq C, D_0$. Consider the sequence of apartments $A_1^{(n)} \to A_1$ where $A_1^{(n)}$ is the apartment containing $C^{(n)}$, $X^{(n)}$ and $v_m^{(n)}$, and the sequence $A_2^{(n)} \to A_2$ where $A_2^{(n)}$ is the apartment containing $D_0^{(n)}$, $X^{(n)}$ and $v_m^{(n)}$ ($X^{(n)}$ is different from $C^{(n)}$ and $D_0^{(n)}$ for sufficiently large $n$). The idea is to apply Lemma~\ref{lemma:key} to these two sequences of apartments simultaneously. We do it for $A_1^{(n)} \to A_1$. Denote by $w_{m+1}^{(n)}, \ldots, w_{2m-1}^{(n)}$ the new vertices of $A_1^{(n)}$ as in Figure~\ref{picture:A2}. The sequence $c_1^{(n)} \to c_1 \not \in A$ is such that $\varphi_n c_1^{(n)} \to \tilde{v}_0$ and, after projecting on $w_{m+1}^{(n)} \to w_{m+1}$ (see Lemma~\ref{lemma:opposite}), there also is a sequence $d^{(n)} \to d \not \in A_1$ with $d^{(n)}$ adjacent to $w_{m+1}^{(n)}$. Moreover, by Claim~\ref{claim:correct} there exists $\eta_1 > 0$ such that $(C_2,C_1)$ is $\eta_1$-correct. We can therefore apply Lemma~\ref{lemma:key} which provides $\eta'_1 > 0$ satisfying the properties stated in this same lemma. In the same way with $A_2^{(n)} \to A_2$, there is $\eta'_2 > 0$ with similar properties.

We now prove that $\eta_X = \min(\eta'_1, \eta'_2)$ is adequate. Suppose that $\varphi_{k(n)}X^{(k(n))} \to \tilde{X}$ for some subsequence $(\varphi_{k(n)})$ of $(\varphi_n)$ and some chamber $\tilde{X}$. If $\tilde{X} \neq C$, then the sequence of apartments $\varphi_{k(n)} A_1^{(k(n))}$ converges to the apartment $\tilde{A}_1$ containing $\tilde{C}$, $\tilde{X}$ and $\tilde{v}_m$. Also, we have in this case $\varphi_{k(n)} d^{(k(n))} \to \tilde{w}_{m+2}$ (the limit of $(\varphi_{k(n)} w_{m+2}^{(k(n))})$) because $\varphi_n c_1^{(n)} \to \tilde{v}_0$. This means that $(\varphi_{k(n)})$ satisfies the three conditions (a), (b) and (c) of Lemma~\ref{lemma:key} for $\eta_1$ and thus $\varphi_{k(n)} R^{(k(n))} \to \tilde{X}$ as soon as $(R^{(n)})$ is a sequence with $R^{(n)} \in B_{v_0^{(n)}}(X^{(n)}, \eta'_1)$ for all $n \in \N$. On the other hand, if $\tilde{X} = C$ then $\tilde{X} \neq D_0$ and we can do the same trick with $A_2^{(n)} \to A_2$ and $\eta'_2$. In either case, we have shown that $\eta_X = \min(\eta'_1, \eta'_2)$ is adequate.

\begin{figure}
\centering
\includegraphics{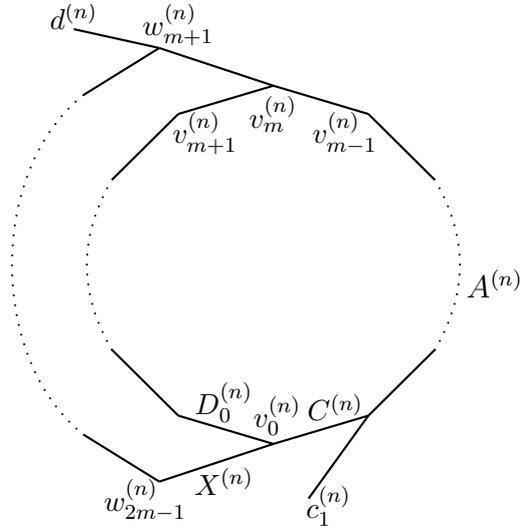}
\caption{Illustration of Claim~\ref{claim:chamv0}.}\label{picture:A2}
\end{figure}

Now if $X = C$, the convergence $\varphi_{k(n)} X^{(k(n))} \to \tilde{D}_0$ cannot happen because $(C,D_0)$ is correct. We can therefore in this case do the same reasoning with only $A_2^{(n)} \to A_2$ and take $\eta_X = \eta'_2$. Similarly, for $X = D_0$ we only consider $A_1^{(n)} \to A_1$ and take $\eta_X = \eta'_1$.
\end{claimproof}

\medskip

To complete the proof, we show that the situation described in Claim~\ref{claim:chamv0} is actually impossible. Indeed, by compactness of $\Cham(v_0)$ we get
$$\Cham(v_0) = \bigcup_{j=1}^r B_{v_0}(X_j, \eta_{X_j})$$
for some $X_1, \ldots, X_r \in \Cham(v_0)$. Passing to a subsequence, we can assume that $\varphi_n X_j^{(n)} \to \tilde{X}_j$ for each $j \in \{1, \ldots, r\}$. Now consider some $S \in \Cham(\tilde{v}_0) \setminus \{\tilde{X}_1, \ldots, \tilde{X}_r\}$ (which is non-empty since each panel is infinite). By a construction similar to the construction of $X^{(n)}$ from $X$, we can build a sequence $S^{(n)} \to S$ with $S^{(n)} \in \Cham(\varphi_n v_0^{(n)})$ for all $n \in \N$. Let $R^{(n)} = \varphi_n^{-1}S^{(n)}$; we can assume that $R^{(n)} \to R \in \Cham(v_0)$. There exists $j \in \{1, \ldots, r\}$ such that $R \in B_{v_0}(X_j, \eta_{X_j})$ and this means that $R^{(n)} \in B_{v_0^{(n)}}(X_j^{(n)}, \eta_{X_j})$ for sufficiently large~$n$. Hence, $(\varphi_n R^{(n)})$ should converge to $\tilde{X}_j$, which is impossible as $\varphi_n R^{(n)} = S^{(n)} \to S \neq \tilde{X}_j$.
\end{proof}

\begin{proof}[Proof of Corollary~\ref{corollary:Epsilon}]
Define $D_0$, $D_1$, $v_0$ and $v_1$ as in Theorem~\ref{theorem:Criterion} and take $\varepsilon > 0$ sufficiently small so that
\begin{enumerate}[(A)]
\item $\overline{B(C, \varepsilon)} \times \overline{B(C', \varepsilon)} \subseteq \{(X,X') \in (\Cham\Delta)^2 \mid X \text{ is opposite } X'\}$,
\item $\rho(C,D_i) \geq 4\varepsilon$ for $i \in \{0, 1\}$,
\item there exists $E_i \in \Cham(v_i)$ such that $\rho(C, E_i) \geq 3 \varepsilon$ and $\rho(D_i, E_i) \geq 3 \varepsilon$ for $i \in \{0,1\}$.
\end{enumerate}
It is obvious that there exists $\varepsilon > 0$ satisfying (B) and (C), and Proposition~\ref{proposition:opposite} tells us that we can choose it so that it also satisfies (A). It is then immediate to see that $L_\varepsilon(C,C')$ is a closed subset of $J_\varepsilon(C,C',\overline{B(C, \varepsilon)},\overline{B(C', \varepsilon)},E_0,E_1)$, the latter being compact by Theorem~\ref{theorem:Criterion}.
\end{proof}

\begin{proof}[Proof of Corollary~\ref{corollary:lc}]
Consider two opposite chambers $C$ and $C'$ of $\Delta$. The set $L_\varepsilon(C,C')$ is a neighbourhood of the identity in $\Auttop(\Delta)$ and is compact for some $\varepsilon > 0$ by Corollary~\ref{corollary:Epsilon}.
\end{proof}

\section{Topological characterization of the Moufang property}\label{section:characterization}

In this last section, our ultimate goal is to prove Theorem~\ref{theorem:Characterization}. The implications (i) $\Rightarrow$ (ii) and (ii) $\Rightarrow$ (iii) are explained in the first two subsections while the most difficult part (iii) $\Rightarrow$ (i) is the subject of the rest of the section.

\subsection{Convergence groups}\label{subsection:convergence}

If $M$ is a topological space and $G$ is a topological group acting continuously on $M$, then the action of $G$ on $M$ is \textbf{proper} if for every compact set $K \subseteq M$, the set $\{g \in G \mid g(K) \cap K \neq \varnothing\}$ is compact. When $M$ is compact, we also say that the action of $G$ on $M$ is \textbf{$n$-proper} if the componentwise action of $G$ on the space
$$M^{(n)} = \{(x_1, \ldots, x_n\} \in M^n \mid x_i \neq x_j \ \ \forall i \neq j\}$$
is proper, where $M^{(n)}$ is equipped with the induced topology from the product topology on $M^{n}$. We finally say that $G$ acts as a \textbf{convergence group} on $M$ if $G$ acts $3$-properly on $M$.

We will be particularly interested in subgroups of the homeomorphism group of a compact metric space $M$ that act as a convergence group on $M$.

\begin{lemma}\label{lemma:lc}
Let $(M, d)$ be a compact metric space and $G$ be a subgroup of $\Homeo(M)$ (equipped with the compact-open topology) acting as a convergence group on $M$. Then $G$ is closed in $\Homeo(M)$ and $G$ is locally compact.
\end{lemma}

\begin{proof}
If $M$ has less than $3$ elements, then $\Homeo(M)$ is finite and hence $G$ is locally compact and closed in $\Homeo(M)$. Now suppose that $|M| \geq 3$.

We first show that $G$ is closed in $\Homeo(M)$. Let $g_n \to g$ be a sequence with $g_n \in G$ for all $n \in \N$ and $g \in \Homeo(M)$. We want to prove that $g \in G$. Take $x_1, x_2, x_3$ three distinct elements of $M$ and consider $\varepsilon > 0$ such that $d(g(x_i),g(x_j)) > 2\varepsilon$ for each $i \neq j$. Define
$$K = \{(x_1,x_2,x_3)\} \cup (\overline{B(g(x_1), \varepsilon)} \times \overline{B(g(x_2), \varepsilon)} \times \overline{B(g(x_3), \varepsilon)}) \subseteq M^{(3)}.$$
This is clearly a compact subset of $M^{(3)}$, and $g_n \in \{h \in G \mid h(K) \cap K \neq \varnothing\}$ for sufficiently large $n$. This set being compact by hypothesis, it is closed in $G$ and $g \in G$. 

We now prove that $G$ is locally compact. Once again, take $x_1, x_2, x_3$ three distinct elements of $M$. Choose $\varepsilon > 0$ such that $d(x_i, x_j) > 3\varepsilon$ for each $i \neq j$ and define
$$K = \{(y_1, y_2, y_3) \in M^3 \mid d(y_i, y_j) \geq \varepsilon \ \ \forall i \neq j\} \subseteq M^{(3)}.$$
This is a compact subset of $M^{(3)}$ and thus $V := \{h \in G \mid h(K) \cap K \neq \varnothing\}$ is compact. But $V$ contains the ball
$$B(\id, \varepsilon) := \{h \in G \mid d(h(x), x) < \varepsilon \ \ \forall x \in M\}.$$
Indeed, if $h \in B(\id, \varepsilon)$, then $(h(x_1),h(x_2),h(x_3)) \in h(K) \cap K$. Therefore, $V$ is a compact neighbourhood of the identity of $G$.
\end{proof}

The notion of convergence group appears in the theory of trees.

\begin{lemma}\label{lemma:tree}
Let $T$ be a locally finite tree. The automorphism group $\Aut(T)$ of $T$ with the topology of pointwise convergence acts as a convergence group on the space of ends $T_\infty$ of $T$.
\end{lemma}

\begin{proof}
The action of $\Aut(T)$ on the set of vertices of $T$ (endowed with the discrete topology) is proper, since $T$ is locally finite. The result then follows from the fact that one can canonically associate a vertex of $T$ to each element of $T_\infty^{(3)}$.
\end{proof}

The reason why we are interested in convergence groups is the existence of the following reciprocal result, due to Carette and Dreesen.

\begin{theorem}[Carette--Dreesen]\label{theorem:Carette}
Let $G$ be a $\sigma$-compact, locally compact, non-compact group acting continuously and transitively as a convergence group on an infinite compact totally disconnected space $M$. Then $G$ acts continuously, properly and faithfully on some locally finite tree $T$ and the spaces $M$ and $T_\infty$ are equivariantly homeomorphic.
\end{theorem}

\begin{proof}
This is a particular case of \cite{Carette}*{Theorem~D}.
\end{proof}

\subsection{Group of projectivities}\label{subsection:projectivities}

If $\pi$ and $\pi'$ are two opposite panels of a compact spherical building $\Delta$, then the projection $\proj_\pi\mid_{\pi'} : \Cham(\pi') \to \Cham(\pi)$ is a homeomorphism (see Corollary~\ref{corollary:homeo}). A map of the form
$$[\pi_0, \pi_1, \ldots, \pi_n] := \proj_{\pi_n} \circ \proj_{\pi_{n-1}} \circ \ldots \circ \proj_{\pi_1} \mid_{\pi_0}\ : \Cham(\pi_0) \to \Cham(\pi_n)$$
where $\pi_0, \pi_1, \ldots, \pi_n$ is a sequence of panels with $\pi_i$ opposite $\pi_{i-1}$ for each $i \in \{1, \ldots, n\}$ is called a \textbf{projectivity} from $\pi_0$ to $\pi_n$. The \textbf{group of projectivities} $\Pi(\pi)$ associated to a panel $\pi$ is the group of all projectivities from $\pi$ to $\pi$ (seen as a subgroup of $\Homeo(\Cham(\pi))$, so that two projectivities acting in the same way on $\Cham(\pi)$ are considered equal).

Now let $X$ be a locally finite thick affine building of rank at least~$3$ and of irreducible type. It is well known (see \cite{Grundhofer}*{Proposition~6.4}) that the spherical building $X_\infty$ at infinity of $X$ can be given a structure of totally disconnected compact building. Also, one can construct for each panel $\pi$ of $X_\infty$ a locally finite thick tree $T(\pi)$, often called \textit{panel-tree}, whose space of ends $T_\infty(\pi)$ (we should actually write $T(\pi)_\infty$) is in bijective correspondence with $\Cham(\pi)$ (see \cite{Ronan}*{Lemma~10.4} or \cite{Weiss}*{Proposition~11.22}). This bijection $\psi : T_\infty(\pi) \to \Cham(\pi)$ is a homeomorphism if we equip $T_\infty(\pi)$ with the usual topology. There also is a natural action of the group of projectivities $\Pi(\pi)$ on $T(\pi)$ and $\psi$ is equivariant under $\Pi(\pi)$. These remarks are summarized in the following proposition.

\begin{proposition}\label{proposition:equivariant}
Let $X$ be a locally finite thick irreducible affine building of rank at least~$3$. For each panel $\pi$ of $X_\infty$, there exist a locally finite thick tree $T(\pi)$ such that the group of projectivities $\Pi(\pi)$ acts on $T(\pi)$ and a homeomorphism $\psi : T_\infty(\pi) \to \Cham(\pi)$ equivariant under $\Pi(\pi)$.
\end{proposition}

\begin{proof}
See the results mentioned in the previous discussion.
\end{proof}

We can now prove the implication (i) $\Rightarrow$ (ii) in Theorem~\ref{theorem:Characterization}.

\begin{proposition}\label{proposition:affine}
Let $X$ be a locally finite thick irreducible affine building of rank at least~$3$. For each panel $\pi$ of $X_\infty$, the closure of the group of projectivities $\Pi(\pi)$ in $\Homeo(\Cham(\pi))$ acts as a convergence group on $\Cham(\pi)$.
\end{proposition}

\begin{proof}
In view of Proposition~\ref{proposition:equivariant}, given a panel $\pi$ of $X_\infty$ there is a tree $T(\pi)$ such that $\Pi(\pi)$ acts on $T(\pi)$ and the spaces $T_\infty(\pi)$ and $\Cham(\pi)$ are equivariantly homeomorphic. We can therefore see $\Pi(\pi)$ as a subgroup of $\Aut(T(\pi))$, and the closure $\overline{\Pi}(\pi)$ of $\Pi(\pi)$ in $\Homeo(T_\infty(\pi))$ remains in $\Aut(T(\pi))$ since $\Aut(T(\pi))$ is closed in $\Homeo(T_\infty(\pi))$. Lemma~\ref{lemma:tree} then states that $\Aut(T(\pi))$ acts as a convergence group on $T_\infty(\pi)$. Since the restriction to a closed subgroup of an action as a convergence group is itself an action as a convergence group, this is also the case of $\overline{\Pi}(\pi)$. Recalling that $T_\infty(\pi)$ and $\Cham(\pi)$ are equivariantly homeomorphic, we see that $\overline{\Pi}(\pi)$ acts as a convergence group on $\Cham(\pi)$.
\end{proof}

\begin{corollary}[Theorem~\ref{theorem:Characterization}, (i) $\Rightarrow$ (ii)]\label{corollary:MoufangCG}
Let $\Delta$ be an infinite thick compact totally disconnected $m$-gon with $m \geq 3$. If $\Delta$ is Moufang, then for each panel $\pi$ of $\Delta$ the closure in $\Homeo(\Cham(\pi))$ of the group of projectivities $\Pi(\pi)$ acts as a convergence group on $\Cham(\pi)$.
\end{corollary}

\begin{proof}
By~\cite{Grundhofer}*{Theorem~1.1}, $\Delta$ is the building at infinity of some locally finite thick irreducible affine building $X$ of rank~$3$, i.e. $\Delta \cong X_\infty$. The result then follows from Proposition~\ref{proposition:affine}.
\end{proof}

To prove the implication (ii) $\Rightarrow$ (iii) in Theorem~\ref{theorem:Characterization}, we need some basic observations about groups acting on a locally finite thick tree $T$ so that the induced action on $T_\infty$ is $2$-transitive. We use a result of Tits to prove the following lemma, but the reader should be aware that it can also be shown without using such a deep result.

\begin{lemma}\label{lemma:hyperbolic}
Let $T$ be a locally finite thick tree and let $G \leq \Aut(T)$. If $G$ acts $2$-transitively on $T_\infty$, then there exists a hyperbolic element in $G$, i.e. some $g \in G$ that stabilizes a bi-infinite geodesic $\ell$ of $T$ and acts non-trivially on $\ell$ by translation.
\end{lemma}

\begin{proof}
Assume for a contradiction that $G$ does not contain any hyperbolic element. Then by~\cite{Titsarbres}*{Proposition~3.4}, $G$ is contained in the stabilizer of a vertex, an edge or an end of $T$. However, none of these three situations is possible since $G$ acts $2$-transitively on $T_\infty$.
\end{proof}

\begin{lemma}\label{lemma:ray}
Let $T$ be a locally finite thick tree and let $G \leq \Aut(T)$. Consider an infinite ray $r = (v_i)_{i \geq 1}$ in $T$ and let $v_0$ and $v'_0$ be two vertices of $T$ adjacent to $v_1$ and different from $v_2$. If $G$ acts $2$-transitively on $T_\infty$, then there exists $g \in G$ such that $g(r) = r$ and $g(v_0) = v'_0$.
\end{lemma}

\begin{proof}
Consider a bi-infinite geodesic $\ell$ of $T$ containing $r$ and $v_0$. By Lemma~\ref{lemma:hyperbolic}, there exists a hyperbolic element in $G$, and by $2$-transitivity of $G$ there even is an element $t \in G$ stabilizing $\ell$ and acting on it non-trivially by translation. Replacing $t$ by $t^{-1}$ if necessary, we can assume that $t(v_0) = v_m$ for some $m > 0$. Now let $\ell'$ be a bi-infinite geodesic containing $r$ and $v'_0$. By $2$-transitivity of $G$, we can conjugate $t$ by an adequate element of $G$ to find $t' \in G$ stabilizing $\ell'$ and such that $t'(v'_0) = v_m$. Then $t'^{-1} t \in G$ fixes $r$ and sends $v_0$ on $v'_0$.
\end{proof}

The implication (ii) $\Rightarrow$ (iii) in Theorem~\ref{theorem:Characterization} can now be shown.

\setcounter{claim}{0}

\begin{proposition}[Theorem~\ref{theorem:Characterization}, (ii) $\Rightarrow$ (iii)]\label{proposition:proj-CG}
Let $\Delta$ be an infinite thick compact totally disconnected $m$-gon with $m \geq 3$ and let $\pi$ be a panel of $\Delta$. Suppose that the closure of the group of projectivities $\Pi(\pi)$ in $\Homeo(\Cham(\pi))$ acts as a convergence group on $\Cham(\pi)$. Then for all $G \leq \Auttop(\Delta)$, the closure of the natural image of $\Stab_G(\pi)$ in $\Homeo(\Cham(\pi))$ acts as a convergence group on $\Cham(\pi)$.
\end{proposition}

\begin{proof}
By~\cite{Knarr}*{Lemma~1.2}, the action of $\Pi(\pi)$ on $\Cham(\pi)$ is $2$-transitive. Now suppose that the closure $\overline{\Pi}(\pi)$ of $\Pi(\pi)$ in $\Homeo(\Cham(\pi))$ acts as a convergence group on $\Cham(\pi)$.

\begin{claim}\label{claim:equivariant}
The group $\overline{\Pi}(\pi)$ acts on a locally finite thick tree $T(\pi)$ such that $T_\infty(\pi)$ and $\Cham(\pi)$ are equivariantly homeomorphic.
\end{claim}

\begin{claimproof}
We want to apply Theorem~\ref{theorem:Carette} and thus check its hypotheses. The group $\overline{\Pi}(\pi)$ is locally compact (Lemma~\ref{lemma:lc}), $\sigma$-compact (it is second-countable by \cite{Arens}*{Theorem~5} and a locally compact second-countable space is always $\sigma$-compact), $\Cham(\pi)$ is infinite (Lemma~\ref{lemma:non-isolated}), compact and totally disconnected, and the action of $\overline{\Pi}(\pi)$ on $\Cham(\pi)$ is continuous and transitive (even $2$-transitive). It only remains to show that $\overline{\Pi}(\pi)$ is non-compact. By contradiction, suppose that $\Pi(\pi)$ is relatively compact in $\Homeo(\Cham(\pi))$. It is therefore also relatively compact in $C(\Cham(\pi), \Cham(\pi))$ and, by the Arzela-Ascoli theorem, $\Pi(\pi)$ is equicontinuous on $\Cham(\pi)$. This is a contradiction with the fact that the action of $\Pi(\pi)$ on $\Cham(\pi)$ is $2$-transitive. By Theorem~\ref{theorem:Carette}, we therefore have a locally finite tree $T(\pi)$ as wanted. This tree may be not thick, but we can actually assume it is. Indeed, the vertices of degree~$1$ can clearly be deleted, as well as the vertices of degree~$2$ by joining their two neighbours. Note that there is no infinite ray in $T(\pi)$ only containing vertices of degree~$2$ (which would prevent the deletion of these vertices) since this would imply the existence of an isolated chamber in $\Cham(\pi)$, which is impossible in view of Lemma~\ref{lemma:non-isolated}.
\end{claimproof}

\medskip

The next claim shows that the structure of $T(\pi)$ is encoded in the group structure of $\overline{\Pi}(\pi)$. For any vertex $v$ of $T(\pi)$, we write $S_v$ for the stabilizer of $v$ in $\overline{\Pi}(\pi)$.

\begin{claim}\label{claim:encoded}
The groups $S_v$ are pairwise distinct and are exactly the maximal compact subgroups $K$ of $\overline{\Pi}(\pi)$ such that $[K : (K \cap H)] \geq 3$ for any other maximal compact subgroup $H$ of $\overline{\Pi}(\pi)$. Moreover, two distinct vertices $v$ and $v'$ of $T(\pi)$ are adjacent if and only if $[S_v : (S_v \cap S_{v'})] \leq [S_v : (S_v \cap S_w)]$ for any vertex $w$ different from $v$.
\end{claim}

\begin{claimproof}
In view of Lemma~\ref{lemma:ray}, the groups $S_v$ are pairwise distinct. Moreover, they are maximal compact subgroups of $\overline{\Pi}(\pi)$ and the only other maximal compact subgroups are the stabilizers of those edges $e$ for which there is an element of $\overline{\Pi}(\pi)$ inverting $e$ (see~\cite{HarmonicAnalysis}*{Theorem~5.2}). If $S_e$ denotes the stabilizer in $\overline{\Pi}(\pi)$ of such an edge $e = [v_1,v_2]$, then $[S_e : (S_e \cap S_{v_1})] = 2$. On the other hand, if $v$ is a vertex of $T(\pi)$ then by Lemma~\ref{lemma:ray} and since $T(\pi)$ is thick we have $[S_v : (S_v \cap S_{v'})] \geq 3$ for $v'$ a vertex different from $v$ and $[S_v : (S_v \cap S_e)] \geq 3$ for $e$ an edge of $T(\pi)$. The first part of the claim is therefore proven.

Now fix a vertex $v$ of $T$. By Lemma~\ref{lemma:ray}, for any other vertex $w$ the equality
$$[S_v : (S_v \cap S_w)] = \deg(v) \cdot \prod_{i=1}^{n-1}(\deg(v_i)-1)$$
holds, where $v, v_1, \ldots, v_{n-1}, w$ is the path in $T(\pi)$ joining $v$ and $w$ (this formula can be obtained by induction on $n$). The second part of the claim follows from this observation.
\end{claimproof}

\medskip

\begin{claim}\label{claim:actiontree}
The stabilizer $\Stab_{\Auttop(\Delta)}(\pi)$ of $\pi$ in $\Auttop(\Delta)$ acts on $T(\pi)$ and the homeomorphism between $T_\infty(\pi)$ and $\Cham(\pi)$ is equivariant under $\Stab_{\Auttop(\Delta)}(\pi)$.
\end{claim}

\begin{claimproof}
First remark that $\Stab_{\Auttop(\Delta)}(\pi)$ normalizes $\Pi(\pi)$ in $\Homeo(\Cham(\pi))$. Indeed, if $[\pi, \pi_1, \ldots, \pi_{n-1}, \pi]$ is a projectivity and $g \in \Stab_{\Auttop(\Delta)}(\pi)$, one can check that
$$g \circ [\pi, \pi_1, \ldots, \pi_{n-1}, \pi] \circ g^{-1} = [\pi, g(\pi_1), \ldots, g(\pi_{n-1}), \pi].$$
In view of Claim~\ref{claim:encoded}, the action by conjugation of $\Stab_{\Auttop(\Delta)}(\pi)$ on $\Pi(\pi)$ induces an action of $\Stab_{\Auttop(\Delta)}(\pi)$ on $T(\pi)$. It remains to show that the homeomorphism $\psi : T_\infty(\pi) \to \Cham(\pi)$ given by Claim~\ref{claim:equivariant} is equivariant under $\Stab_{\Auttop(\Delta)}(\pi)$. Consider $g \in \Stab_{\Auttop(\Delta)}(\pi)$, $C \in \Cham(\pi)$ and $(v_i)_{i \geq 1}$ a ray in $T(\pi)$ representing the end corresponding to $C$. We want to show that the ray $(g(v_i))_{i \geq 1}$ represents the end corresponding to $g(C)$. To do so, remark that
$$\bigcup_{k=1}^\infty \bigcap_{i=k}^\infty S_{v_i}$$
fixes $C$ but does not fix any chamber different from $C$ (this follows from Lemma~\ref{lemma:ray}). Now we also have
$$g\left(\bigcup_{k=1}^\infty \bigcap_{i=k}^\infty S_{v_i}\right)g^{-1} = \bigcup_{k=1}^\infty \bigcap_{i=k}^\infty S_{g(v_i)} =: R$$
by definition of the action of $\Stab_{\Auttop(\Delta)}(\pi)$ on $T(\pi)$. But $R$ fixes $g(C)$ and also fixes the chamber corresponding to the end represented by the ray $(g(v_i))_{i \geq 1}$. Since $R$ can only fix one chamber, the claim stands proven.
\end{claimproof}

\medskip

In view of Claim~\ref{claim:actiontree} and since $\Aut(T(\pi))$ is closed in $\Homeo(T_\infty(\pi))$, for any $G \leq \Auttop(\Delta)$ the closure of the image of $\Stab_G(\pi)$ in $\Homeo(\Cham(\pi))$ can be seen as a (closed) subgroup of $\Aut(T(\pi))$. The conclusion thus follows from Lemma~\ref{lemma:tree}, as the restriction to a closed subgroup of an action as a convergence group is itself an action as a convergence group.
\end{proof}

\subsection{Conventions for the rest of the section}\label{subsection:conventions}

The rest of this section is dedicated to the proof of the implication (iii) $\Rightarrow$ (i) in Theorem~\ref{theorem:Characterization}. Unless otherwise stated, we will from now on study an infinite thick compact totally disconnected $m$-gon $\Delta$ with $m \geq 3$. We will also assume that we are in presence of a strongly transitive closed subgroup $G$ of $\Auttop(\Delta)$ such that $\Delta$ satisfies the \textit{property $\CG_G$}:

\begin{definition}
Let $\Delta$ be a compact polygon and let $G$ be a subgroup of $\Auttop(\Delta)$. We will say that $\Delta$ has the \textbf{property $\CG_G$} if for each vertex (i.e. panel) $v$ of $\Delta$, the closure of the image of $\Stab_{G}(v)$ in $\Homeo(\Cham(v))$ acts as a convergence group on $\Cham(v)$.
\end{definition}

Our ultimate goal is to prove that, under these conditions, $\Delta$ has the Moufang property. We can already directly use Theorem~\ref{theorem:Carette} to construct a tree associated to each vertex of $\Delta$.

\begin{proposition}\label{proposition:tree}
Let $\Delta$ be an infinite thick compact totally disconnected $m$-gon with $m \geq 3$ and let $G$ be a strongly transitive subgroup of $\Auttop(\Delta)$ satisfying property $\CG_G$. Then for each vertex $v$ of $\Delta$ there exists a locally finite thick tree $T(v)$ such that $\Stab_{G}(v)$ acts on $T(v)$ and the spaces $T_\infty(v)$ and $\Cham(v)$ are equivariantly homeomorphic.
\end{proposition}

\begin{proof}
The property $\CG_G$ tells us that the group $\overline{i(\Stab_G(v))}$ where $i$ is the natural map from $\Stab_{G}(v)$ to $\Homeo(\Cham(v))$ acts as a convergence group on $\Cham(v)$. All the hypotheses of Theorem~\ref{theorem:Carette} are then satisfied. Indeed, $\overline{i(\Stab_G(v))}$ is locally compact (Lemma~\ref{lemma:lc}) and $\sigma$-compact (it is second-countable by \cite{Arens}*{Theorem 5} and a locally compact second-countable space is always $\sigma$-compact), $\Cham(v)$ is infinite (Lemma~\ref{lemma:non-isolated}), compact and totally disconnected, and the action of $\overline{i(\Stab_G(v))}$ on $\Cham(v)$ is continuous and transitive (even $2$-transitive, see Lemma~\ref{lemma:strongtransitivity}). It only remains to show that $\overline{i(\Stab_G(v))}$ is non-compact. Suppose for a contradiction that $i(\Stab_G(v))$ is relatively compact in $\Homeo(\Cham(v))$. It is therefore also relatively compact in $C(\Cham(v), \Cham(v))$ and, by the Arzela-Ascoli theorem, $i(\Stab_G(v))$ is equicontinuous on $\Cham(v)$. This is a contradiction with the fact that the action of $i(\Stab_G(v))$ on $\Cham(v)$ is $2$-transitive.

Hence, there is a locally finite tree $T(v)$ such that $\overline{i(\Stab_G(v))}$ acts on $T(v)$ and so that $T_\infty(v)$ and $\Cham(v)$ are equivariantly homeomorphic. The group $\Stab_G(v)$ obviously acts on $T(v)$ too, and the homeomorphism remains equivariant under $\Stab_G(v)$. We can also assume that $T(v)$ is thick since the vertices of degree $\leq 2$ can be deleted (for those of degree~$2$, the two neighbours have to be joined by an edge). Note that there is no infinite ray only containing vertices of degree $2$ since it would contradict Lemma~\ref{lemma:non-isolated}.
\end{proof}

\subsection{Automorphisms fixing an apartment}

In this subsection, we focus on the fixator $H = \Fix_G(A)$ of an apartment $A$ of $\Delta$ (where $\Delta$ and $G$ are as described in Subsection~\ref{subsection:conventions}). By Proposition~\ref{proposition:tree}, for each vertex $v$ of $A$ there exists a tree $T(v)$ whose space of ends corresponds to $\Cham(v)$. The group $H$ acts simultaneously on all these trees. Moreover, for each vertex $v$ of $A$, $H$ fixes the two chambers of $A$ having vertex $v$, which means that it fixes the two corresponding ends of $T(v)$. In this tree, $H$ thus stabilizes the bi-infinite geodesic between these two ends, which we denote by $\ell_v$. We therefore have for each vertex $v$ of $A$ an induced action of $H$ on $\ell_v$.

Our goal is now to show that for every vertex $v$ of $A$, there exists some element of $H$ fixing $\ell_{v}$ and $\ell_{v'}$ pointwise where $v'$ is the vertex of $A$ opposite $v$ but acting non-trivially on every other $\ell_{w}$. We will actively use this result in the last subsection to prove that $\Delta$ is Moufang.

As a first step in this direction, we show that for each vertex $v$ of $A$, there is an element of $H$ that acts non-trivially on $\ell_v$.

\begin{lemma}\label{lemma:notfix}
Let $\Delta$ be an infinite thick compact totally disconnected $m$-gon with $m \geq 3$ and let $G$ be a strongly transitive subgroup of $\Auttop(\Delta)$ satisfying property $\CG_G$. Let $A$ be an apartment of $\Delta$, $v$ be a vertex of $A$ and $H = \Fix_G(A)$. There is an element of $H$ that acts non-trivially on $\ell_v$.
\end{lemma}

\begin{proof}
Let $v'$ be the vertex of $A$ opposite $v$ and let $L = \Stab_G(v,v')$. Since $G$ is strongly transitive, the natural image $\tilde{L}$ of $L$ in $\Aut(T(v))$ acts $2$-transitively on $T_\infty(v)$ (see Lemma~\ref{lemma:strongtransitivity}). By Lemma~\ref{lemma:hyperbolic}, there is a hyperbolic element in $\tilde{L}$ and hence an element that stabilizes $\ell_v$ but does not fix it pointwise. It corresponds to an element of $H$ that does not fix $\ell_v$ pointwise.
\end{proof}

To go further in the discussion, we denote by $v_0, \ldots, v_{2m-1}$ the vertices of $A$ in the natural order and distinguish three types of possible actions of an element $t \in H$ around a vertex $v_i$:
\begin{enumerate}[(a)]
\item If $t$ fixes $\ell_{v_i}$ pointwise, we will say that $t$ acts \textbf{trivially around} $v_i$.
\item If $t$ translates $\ell_{v_i}$ toward the end representing the chamber of $A$ having vertex $v_{i+1}$, we will say that $t$ acts \textbf{positively around} $v_i$.
\item If $t$ translates $\ell_{v_i}$ toward the end representing the chamber of $A$ having vertex $v_{i-1}$, we will say that $t$ acts \textbf{negatively around} $v_i$.
\end{enumerate}

An example of such an action of an element $t \in H$ on $A$ is shown in Figure~\ref{picture:action}. An arrow in the anticlockwise direction means that the action is positive while an arrow in the clockwise direction means that it is negative. In this figure, the action is positive around $v_0, v_1$ and $v_2$ and negative around $v_3, v_4$ and $v_5$.

The three types of action can easily be spotted by observing how $t^n$ acts on $\Delta$ when $n$ goes to infinity. Indeed, denoting by $C_{+}$ the chamber having vertices $v_i$ and $v_{i+1}$, by $C_{-}$ the chamber having vertices $v_i$ and $v_{i-1}$, and by $C$ any chamber having vertex $v_i$ different from $C_+$ and $C_-$, we get the following characterizations.
\begin{enumerate}[(a)]
\item $t$ acts trivially around $v_i$ $\Leftrightarrow$ $t^n(C)$ does not accumulate at $C_+$ nor at $C_-$.
\item $t$ acts positively around $v_i$ $\Leftrightarrow$ $t^n(C) \to C_+$.
\item $t$ acts negatively around $v_i$ $\Leftrightarrow$ $t^n(C) \to C_-$.
\end{enumerate}
\begin{figure}
\centering
\includegraphics{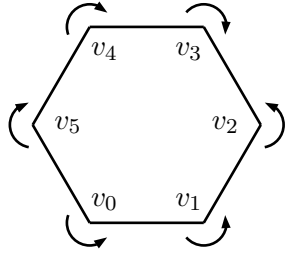}
\caption{Representation of an action.}\label{picture:action}
\end{figure}

These characterizations already allow us to prove the following lemma giving a relationship between the action of an element $t$ around $v_i$ and its action around $v_{i+m}$, the opposite vertex in $A$. The action represented in Figure~\ref{picture:action} is coherent with this information.

\begin{lemma}\label{lemma:oppositeaction}
The action of an element $t \in H$ around a vertex $v_i$ of $A$ is opposite to its action around $v_{i+m}$: $t$ either acts trivially around the two vertices, or it acts positively around one and negatively around the other.
\end{lemma}

\begin{proof}
Let $C \not \in A$ be a chamber having vertex $v_i$. Then $\proj_{v_{i+m}}(C)$ is a chamber having vertex $v_{i+m}$ and $t^n(\proj_{v_{i+m}}(C)) = \proj_{v_{i+m}}(t^n(C))$. By Proposition~\ref{proposition:projections}, the action of $t$ around $v_i$ is opposite to its action around $v_{i+m}$.
\end{proof}

We can now prove that $H$ always contains some particular elements.

\setcounter{claim}{0}

\begin{proposition}\label{proposition:particular}
Let $\Delta$ be an infinite thick compact totally disconnected $m$-gon with $m \geq 3$ and let $G$ be a strongly transitive subgroup of $\Auttop(\Delta)$ satisfying property $\CG_G$. Let $A$ be an apartment of $\Delta$, $v_0, \ldots, v_{2m-1}$ be the vertices of $A$ and $H = \Fix_G(A)$. For each $i \in \{0, \ldots, 2m-1\}$, there is an element of $H$ that acts positively around all vertices of $A$ closer to $v_i$ than to $v_{i+1}$ and negatively around all vertices of $A$ closer to $v_{i+1}$ than to $v_i$ (indices being taken modulo $2m$).
\end{proposition}

\begin{proof}
We divide the proof into three claims.

\begin{claim}\label{claim:somewhere}
There exists an element of $H$ that acts positively around $v_s$ and negatively around $v_{s+1}$ for some $s \in \{0, \ldots, 2m-1\}$.
\end{claim}

\begin{claimproof}
In view of Lemma~\ref{lemma:notfix}, there is an element of $H$ acting positively around at least one vertex of $A$. We consider $t \in H$ an element acting positively around a maximal number of consecutive vertices of $A$, and denote by $v_r, v_{r+1}, \ldots, v_s$ these vertices. Note that $t$ does not act positively around all vertices of $A$ in view of Lemma~\ref{lemma:oppositeaction}. Hence, $t$ acts negatively or trivially around $v_{s+1}$. Suppose that $t$ acts trivially around $v_{s+1}$. By Lemma~\ref{lemma:notfix}, there is $t' \in H$ acting non-trivially around $v_{s+1}$. Replacing $t'$ by $t'^{-1}$ if necessary, we can assume that $t'$ acts positively around $v_{s+1}$. Then for sufficiently large $n \in \N$, $t^n t' \in H$ acts positively around $v_r, \ldots, v_s, v_{s+1}$ which contradicts the maximality of $t$. This means that $t$ acts negatively around $v_{s+1}$, and the claim stands proven.
\end{claimproof}

\medskip

\begin{claim}\label{claim:small}
For each $i \in \{0, \ldots, 2m-1\}$, there exists an element of $H$ that acts positively around $v_i$ and negatively around $v_{i+1}$.
\end{claim}

\begin{claimproof}
By Claim~\ref{claim:somewhere}, there is $t \in H$ acting positively around $v_s$ and negatively around $v_{s+1}$ for some $s \in \{0, \ldots, 2m-1\}$. Now take $i \in \{0, \ldots, 2m-1\}$. Since $G$ is strongly transitive, there is an element $g \in H$ sending the chamber with vertices $v_s$ and $v_{s+1}$ to the chamber having vertices $v_i$ and $v_{i+1}$. It can be directly checked that $gtg^{-1} \in H$ acts positively around $v_i$ and negatively around $v_{i+1}$.
\end{claimproof}

\medskip

\begin{claim}\label{claim:particular}
Let $s \in \{0, \ldots, 2m-1\}$ and let $t \in H$ be an element acting positively around $v_s$ and negatively around $v_{s+1}$. Then $t$ acts positively around all vertices of $A$ closer to $v_s$ than to $v_{s+1}$ and negatively around all vertices of $A$ closer to $v_{s+1}$ than to $v_s$.
\end{claim}

\begin{claimproof}
\begin{figure}[b!]
\centering
\includegraphics{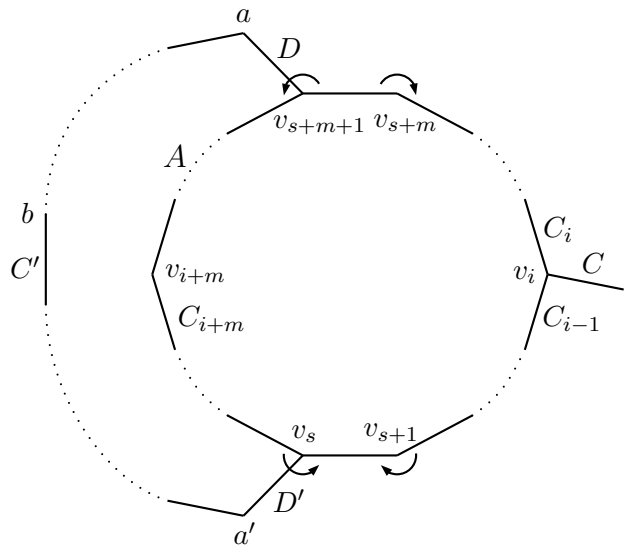}
\caption{Illustration of Claim~\ref{claim:particular}.}\label{picture:particular}
\end{figure}
By Lemma~\ref{lemma:oppositeaction}, $t$ acts negatively around $v_{s+m}$ and positively around $v_{s+m+1}$ (see Figure~\ref{picture:particular}). We now want to show that $t$ acts negatively around each vertex between $v_{s+1}$ and $v_{s+m}$ and positively around all other vertices. We therefore consider some vertex $v_i$ with $s+1 < i < s+m$ and show that the action around $v_i$ is negative (and it will follow that the action around $v_{i+m}$ is positive). Denoting by $C_j$ the chamber having vertices $v_j$ and $v_{j+1}$ for $j \in \{0, \ldots, 2m-1\}$, it suffices to find some chamber $C \not \in A$ having vertex $v_i$ and such that $t^n(C) \to C_{i-1}$. To get one, consider a chamber $D \not \in A$ having vertex $v_{s+m+1}$ and a chamber $D' \not \in A$ having vertex $v_s$. Let also $a$ (resp. $a'$) be the vertex of $D$ (resp. $D'$) different from $v_{s+m+1}$ (resp. $v_s$). The vertices $a$ and $a'$ are almost opposite while $t^n(a) \to v_{s+m+2}$ and $t^n(a') \to v_{s+1}$ with $v_{s+m+2}$ and $v_{s+1}$ almost opposite. Hence, $t^n([a, a']) \to [v_{s+m+2}, v_{s+1}]$ (Lemma~\ref{lemma:gallery}). Now consider $b$ the first vertex of $[a,a']$ opposite $v_i$ and $C'$ the chamber just after $b$ in this new gallery from $v_{s+m+1}$ to $v_s$ (see Figure~\ref{picture:particular}). Clearly $t^n(b) \to v_{i+m}$ and $t^n(C') \to C_{i+m}$. Define finally $C = \proj_{v_i}(C')$ to get $t^n(C) \to C_{i-1}$ as wanted (see Proposition~\ref{proposition:projections}). Note that $C \neq C_{i-1}$ since the vertex of $C'$ different from $b$ is opposite $v_{i-1}$.
\end{claimproof}

\medskip

The proposition follows from Claim~\ref{claim:small} and Claim~\ref{claim:particular}.
\end{proof}

The desired result can finally be shown.

\begin{proposition}\label{proposition:particular2}
Let $\Delta$ be an infinite thick compact totally disconnected $m$-gon with $m \geq 3$ and let $G$ be a strongly transitive subgroup of $\Auttop(\Delta)$ satisfying property $\CG_G$. Let $A$ be an apartment of $\Delta$, $v_0, \ldots, v_{2m-1}$ be the vertices of $A$ and $H = \Fix_G(A)$. For each $i \in \{0, \ldots, m-1\}$, there is an element of $H$ that acts trivially around $v_i$ and $v_{i+m}$ but non-trivially around all other vertices of $A$.
\end{proposition}

\begin{proof}
Fix $i \in  \{0,\ldots,m-1\}$. By Proposition~\ref{proposition:particular}, there exists $t_1 \in H$ acting positively around all vertices of $A$ closer to $v_{i-1}$ than to $v_i$ and negatively around all vertices of $A$ closer to $v_i$ than to $v_{i-1}$. There also exists $t_2 \in H$ acting positively around all vertices of $A$ closer to $v_i$ than to $v_{i+1}$ and negatively around all vertices of $A$ closer to $v_{i+1}$ than to $v_i$. Denote by $T_1 \in \N^*$ (resp. $T_2 \in \N^*$) the length of the translation of $\ell_{v_i}$ induced by $t_1$ (resp. $t_2$). Now consider the element $t = t_1^{T_2} t_2^{T_1} \in H$. Clearly, $t$ acts trivially around $v_i$ (and hence around $v_{i+m}$) and non-trivially around all other vertices of $A$.
\end{proof}

\subsection{Proof of the characterization of the Moufang property}

The proof of implication (iii) $\Rightarrow$ (i) in Theorem~\ref{theorem:Characterization} will use the theory of contraction groups. If $G$ is a topological group and $t \in G$, the \textbf{contraction group} associated to $t$ is the subgroup of $G$ defined by
$$\Con_G(t) = \{g \in G \mid t^n g t^{-n} \to e\}$$
where $e$ is the identity element of $G$. The \textbf{parabolic subgroup} associated to $t$ is the subgroup of $G$ defined by
$$\Par_G(t) = \{g \in G \mid \{t^n g t^{-n}\}_{n \in \N} \text{ is relatively compact in } G\}.$$

The following lemma is immediate.

\begin{lemma}\label{lemma:inclusion}
Let $f : G \to H$ be a continuous homomorphism of topological groups and let $t \in G$. We have the following inclusions.
$$f(\Con_G(t)) \subseteq \Con_H(f(t)) \ \ \text{ and } \ \ f(\Par_G(t)) \subseteq \Par_H(f(t))$$
\end{lemma}

\begin{proof}
Using the continuity of $f$, we directly get $f(\Con_G(t)) \subseteq \Con_H(f(t))$. Similarly, the inclusion $f(\Par_G(t)) \subseteq \Par_H(f(t))$ follows from the fact that a continuous image of a compact set is compact.
\end{proof}

In the case where $G$ is the automorphism group of a tree, we can compute the contraction groups and the parabolic subgroups associated to some particular elements.

\begin{lemma}\label{lemma:treegroups}
Let $T$ be a locally finite tree and let $t$ be an element of $\Aut(T)$ stabilizing a bi-infinite geodesic $(a,b)$ in $T$ (where $a, b \in T_\infty$).
\begin{enumerate}[(i)]
\item If $t$ fixes $(a, b)$ pointwise, then $\Con_{\Aut(T)}(t) = \{e\}$.
\item If $t$ translates $(a, b)$ toward $b$, then $\Par_{\Aut(T)}(t) = \Fix_{\Aut(T)}(a)$.
\end{enumerate}
\end{lemma}

\begin{proof}
Point (i) directly comes from the fact that $\{t^n\}_{n \in \N}$ is relatively compact in $\Aut(T)$ when $t$ fixes $(a,b)$ pointwise. The proof of (ii) is given in~\cite{Caprace}*{Lemma~2.3}.
\end{proof}

Finally, the following theorem is due to Baumgartner and Willis.

\begin{theorem}[Baumgartner--Willis]\label{theorem:Baumgartner}
Let $G$ be a totally disconnected locally compact group and let $t \in G$. Then
$$\Par_G(t) = \Con_G(t) \cdot (\Par_G(t) \cap \Par_G(t^{-1})).$$
\end{theorem}

\begin{proof}
See~\cite{Baumgartner}*{Corollary~3.17}.
\end{proof}

We are now able to prove the implication (iii) $\Rightarrow$ (i) in Theorem~\ref{theorem:Characterization}.

\setcounter{claim}{0}

\begin{theorem}[Theorem~\ref{theorem:Characterization}, (iii) $\Rightarrow$ (i)]\label{theorem:Moufang}
Let $\Delta$ be an infinite thick compact totally disconnected $m$-gon with $m \geq 3$ and let $G$ be a closed strongly transitive subgroup of $\Auttop(\Delta)$ satisfying property $\CG_G$. Then $\Delta$ is Moufang.
\end{theorem}

\begin{proof}
Fix some root $\alpha$ of $\Delta$ and denote by $v_0, v_1, \ldots, v_m$ its vertices in the natural order. We want to prove that the root group
$$U_\alpha = \{g \in \Aut(\Delta) \mid g \text{ fixes every chamber having a panel in } \alpha \setminus \partial\alpha\}$$
acts transitively on the set of apartments containing $\alpha$. This is equivalent to saying that $U_\alpha$ acts transitively on $\Cham(v_0) \setminus \{C\}$ where $C$ is the chamber having vertices $v_0$ and $v_1$. First remark that $V_0 := \Fix_G(\alpha)$ acts transitively on $\Cham(v_0) \setminus \{C\}$ since $G$ is strongly transitive (see~Lemma~\ref{lemma:strongtransitivity}). We will therefore proceed by induction, by defining successively
\begin{align*}
V_1 &= \Ker(V_0 \curvearrowright \Cham(v_1)), \\
V_2 &= \Ker(V_1 \curvearrowright \Cham(v_2)), \\
\vdots \hspace{0.1cm} & \hspace{2cm} \vdots \\
V_{m-1} &= \Ker(V_{m-2} \curvearrowright \Cham(v_{m-1})).
\end{align*}
In this way, $V_1$ is the fixator in $G$ of $\alpha$ and every chamber having vertex $v_1$, $V_2$ is the fixator in $G$ of $\alpha$ and every chamber having vertex $v_1$ or $v_2$, and so on until $V_{m-1}$ which is exactly $U_\alpha \cap G$. The action of $V_0$ on $\Cham(v_0) \setminus \{C\}$ is transitive and we would like to prove that the action of $V_{m-1}$ on this set remains transitive. It thus suffices to show that if the action of $V_i$ is transitive, then the action of $V_{i+1}$ is transitive too.

Suppose that the action of $V_i$ on $\Cham(v_0) \setminus \{C\}$ is transitive for some $i \in \{0, \ldots, m-2\}$. Consider $A$ any apartment containing $\alpha$ and $H = \Fix_G(A)$. By Proposition~\ref{proposition:particular2}, there exists $t \in H$ fixing $\ell_{v_{i+1}}$ pointwise but not $\ell_{v_j}$ for $j \in \{0, \ldots, i\}$. Let $a$ be the end of $T(v_0)$ corresponding to the chamber $C$ and $b$ be the end of $T(v_0)$ corresponding to the other chamber of $A$ having vertex $v_0$. Replacing $t$ by $t^{-1}$ if necessary, we can assume that $t$ acts negatively around $v_0$, i.e. that it translates $\ell_{v_0} = (a,b)$ toward $b$. Now consider $W_i$ the semidirect product in $G$ of $V_i$ and $\langle t \rangle$ (this is well defined as $H$ normalizes $V_i$). To prove that the action of $V_{i+1}$ on $\Cham(v_0) \setminus \{C\}$ remains transitive, we show that $\Con_{W_i}(t)$ is included in $V_{i+1}$ and that its action on the set remains transitive.

\medskip

\begin{claim}\label{claim:ParW}
The equality $\Par_{W_i}(t) = W_i$ holds.
\end{claim}

\begin{claimproof}
We prove the claim by contradiction. Suppose that there is some element $g \in W_i$ such that $\{t^n g t^{-n}\}_{n \geq 0}$ is not relatively compact in $W_i$. Since $W_i$ is closed in $\Auttop(\Delta)$, it is not relatively compact in $\Auttop(\Delta)$ either. By Proposition~\ref{proposition:closer} with $S = \{t^n g t^{-n}\}_{n \geq 0}$, there exist two sequences in $\Vertex \Delta$ which are collapsed by a subsequence of $(t^ngt^{-n})$ or $(t^ng^{-1}t^{-n})$. Replacing $g$ by $g^{-1}$ if necessary, we can assume that it is a subsequence of $(t^ngt^{-n})$. We would then like to use Proposition~\ref{proposition:centerok} with the sequence of apartments $A \to A$. Let $v_{m+1}, \ldots, v_{2m-1}$ be the vertices of $A$ outside $\alpha$ (in the natural order) and $D$ the chamber having vertices $v_0$ and $v_{2m-1}$. Passing to a subsequence, we can assume that $t^ngt^{-n} D \to \tilde{D}$ for some chamber $\tilde{D}$. If $\tilde{D} \neq C$, then $t^ngt^{-n} A \to \tilde{A}$ where $\tilde{A}$ is the unique apartment containing $\alpha$ and $\tilde{D}$ and Proposition~\ref{proposition:centerok} gives two sequences centered at $v_k \to v_k$ for some $k \in \{0, \ldots, 2m-1\}$ and collapsed by a subsequence of $(t^ngt^{-n})$. By Lemma~\ref{lemma:opposite}, we can assume that $k \in \{0, \ldots, m-1\}$. On the other hand, if $\tilde{D} = C$, then $v_1 \to v_1$ and $v_{2m-1} \to v_{2m-1}$ are sequences centered at $v_0 \to v_0$ and collapsed by $(t^ngt^{-n})$.

In either case, there are two sequences centered at $v_k \to v_k$ for some $k \in \{0, \ldots, m-1\}$ and collapsed by a subsequence of $(t^ngt^{-n})$. We introduce the notation (for $j \in \{0, \ldots, m\}$)
$$\varphi_j : W_i \to \Aut(T(v_j))$$ for the natural map from $W_i \subseteq \Stab_{\Auttop(\Delta)}(v_j)$ to $\Aut(T(v_j))$. For each $j \in \{1, \ldots, m-~1\}$, since $t^ngt^{-n}$ translates $\ell_{v_j}$ exactly as $g$ does for all $n \in \N$, the set $\{\varphi_j(t^ngt^{-n})\}_{n \geq 0}$ is relatively compact in $\Aut(T(v_j))$. It is also relatively compact in $\Homeo(T_\infty(v_j))$ and hence equicontinuous on $T_\infty(v_j)$ by the Arzela-Ascoli theorem. It is therefore impossible to have two sequences centered at $v_j \to v_j$ and collapsed by a subsequence of $(t^ngt^{-n})$ for $j \in \{1, \ldots, m-1\}$. But it is not possible either to have them centered at $v_0 \to v_0$. Indeed, the existence of two such sequences would imply that $\varphi_0(g) \not \in \Par_{\Aut(T(v_0))}(\varphi_0(t)) = \Fix_{\Aut(T(v_0))}(a)$ (see Lemma~\ref{lemma:treegroups} (ii)) which is impossible since $g \in W_i$ fixes $a$. The claim is therefore proven.
\end{claimproof}

\medskip

\begin{claim}\label{claim:transitiveCon}
$\Con_{W_i}(t)$ acts transitively on $\Cham(v_0) \setminus \{C\}$.
\end{claim}

\begin{claimproof}
Applying Theorem~\ref{theorem:Baumgartner} to $t \in W_i$ (note that $W_i$ is locally compact as a closed subgroup of $\Auttop(\Delta)$ which is locally compact (Corollary~\ref{corollary:lc}) and totally disconnected since it acts continuously and faithfully on the totally disconnected space $\Cham\Delta$), we get
$$\Par_{W_i}(t) = \Con_{W_i}(t) \cdot (\Par_{W_i}(t) \cap \Par_{W_i}(t^{-1})).$$
Now by Lemma~\ref{lemma:inclusion},
$$\varphi_0(\Par_{W_i}(t)) \subseteq \varphi_0(\Con_{W_i}(t)) \cdot (\Par_{\Aut(T(v_0))}(\varphi_0(t)) \cap \Par_{\Aut(T(v_0))}(\varphi_0(t^{-1}))).$$
Claim~\ref{claim:ParW} together with Lemma~\ref{lemma:treegroups} (ii) finally gives
$$\varphi_0(W_i) \subseteq \varphi_0(\Con_{W_i}(t)) \cdot \Fix_{\Aut(T(v_0))}(a, b).$$
But $V_i$ acts transitively on $\Cham(v_0) \setminus \{C\}$ and so does $W_i$, which means that $\varphi_0(W_i)(b) = T_\infty(v_0) \setminus \{a\}$. Hence,
$$T_\infty(v_0) \setminus \{a\} = \varphi_0(W_i)(b) \subseteq \varphi_0(\Con_{W_i}(t))(b)$$
and $\Con_{W_i}(t)$ also acts transitively on $\Cham(v_0) \setminus \{C\}$.
\end{claimproof}

\medskip

\begin{claim}\label{claim:inclusionCon}
The inclusion $\Con_{W_i}(t) \subseteq V_{i+1}$ holds.
\end{claim}

\begin{claimproof}
It is not even clear for the moment that $\Con_{W_i}(t)$ is included in $V_i$. If $i = 0$, then $t \in H \subseteq V_0$ and $W_0 = V_0$ so this is true. If $i > 0$, then $t \not \in V_i$ (that is actually why we introduced the semidirect product $W_i$) but we can prove that $\Con_{W_i}(t) \subseteq V_i$. Indeed, suppose by contradiction that there is an element of $\Con_{W_i}(t)$ that is not in $V_i$, i.e. of the form $xt^k$ with $x \in V_i$ and $k \in \Z_0$. This means that $t^n x t^{k-n} \to e$ in $W_i$ when $n$ goes to infinity. But $t$ does not fix $\ell_{v_i}$ pointwise while $x$ does as it acts trivially on $\Cham(v_i)$. Hence, $t^n x t^{k-n}$ does not fix $\ell_{v_i}$ pointwise for each $n \in \N$ and $t^n x t^{k-n} \not\to e$. This proves $\Con_{W_i}(t) \subseteq V_i$. Finally, since
$$\varphi_{i+1}(\Con_{W_i}(t)) \subseteq \Con_{\Aut(T(v_{i+1}))}(\varphi_{i+1}(t)) = \{e\}$$
by Lemma~\ref{lemma:treegroups} (i), we get $\Con_{W_i}(t) \subseteq V_{i+1}$.
\end{claimproof}

\medskip

As explained above, the conclusion follows from Claim~\ref{claim:transitiveCon} and Claim~\ref{claim:inclusionCon}.
\end{proof}

We finally give the proof of Corollary~\ref{corollary:affine}.

\begin{proof}[Proof of Corollary~\ref{corollary:affine}]
Let $X$ be a locally finite thick affine building of rank~$3$ and of irreducible type which is strongly transitive. It is clear that $\Aut(X)$ is a closed subgroup of $\Auttop(X_\infty)$ and, by Proposition~\ref{proposition:affine} and Proposition~\ref{proposition:proj-CG}, $X_\infty$ has the property $\CG_{\Aut(X)}$. Moreover, $\Aut(X)$ is strongly transitive on $X$ and hence on $X_\infty$. By Theorem~\ref{theorem:Moufang}, $X_\infty$ is Moufang. In other words, $X$ is a Bruhat-Tits building. Those buildings have been classified by Bruhat and Tits; this classification is the subject of an important part of the book~\cite{Weiss}.
\end{proof}


\begin{bibdiv}
\begin{biblist}

\bib{Arens}{article}{
author = {Arens, Richard F.},
title = {A topology for spaces of transformations},
journal = {Ann. of Math. (2)},
volume = {47},
number = {3},
year = {1946},
pages = {480--495}
}

\bib{Baumgartner}{article}{
author = {Baumgartner, Udo},
author = {Willis, George A.},
title = {Contraction groups and scales of automorphisms of totally disconnected locally compact groups},
journal = {Israel J. Math.},
volume = {142},
year = {2004},
pages = {221--248}
}

\bib{Buekenhout}{article}{
author = {Buekenhout, Francis},
author = {Van Maldeghem, Hendrik},
title = {Finite distance-transitive generalized polygons},
journal = {Geom. Dedicata},
year = {1994},
volume = {52},
pages = {41--51},
number = {1}
}

\bib{Burns}{article}{
author = {Burns, Keith},
author = {Spatzier, Ralf},
title = {On topological Tits building and their classification},
journal = {Publ. Math. Inst. Hautes \'Etudes Sci.},
number = {65},
year = {1987},
pages = {5--34}
}

\bib{Burns2}{article}{
author = {Burns, Keith},
author = {Spatzier, Ralf},
title = {Manifolds of nonpositive curvature and their buildings},
journal = {Publ. Math. Inst. Hautes \'Etudes Sci.},
number = {65},
year = {1987},
pages = {35--59}
}

\bib{Caprace}{article}{
author = {Caprace, Pierre-Emmanuel},
author = {De Medts, Tom},
title = {Trees, contraction groups, and Moufang sets},
journal = {Duke Math. J},
volume = {162},
number = {13},
year = {2013},
pages = {2413--2449}
}

\bib{CapraceMonod}{article}{
author = {Caprace, Pierre-Emmanuel},
author = {Monod, Nicolas},
title = {An indiscrete Bieberbach theorem: from amenable $\mathrm{CAT(0)}$ groups to Tits buildings},
journal = {J. \'Ec. polytech. Math.},
volume = {2},
year = {2015},
pages = {333--383}
}

\bib{Carette}{article}{
author = {Carette, Mathieu},
author = {Dreesen, Denis},
title = {Locally compact convergence groups and $n$-transitive actions},
journal = {Math. Z.},
volume = {278},
number = {3--4},
pages = {795-827},
year = {2014}
}

\bib{HarmonicAnalysis}{book}{
author = {Figà-Talamanca, Alessandro},
author = {Nebbia, Claudio},
title = {Harmonic analysis and representation theory for groups acting on homogeneous trees},
series = {London Math. Soc. Lecture Note Ser.},
volume = {162},
publisher = {Cambridge University Press},
year = {1991},
place = {Cambridge}
}

\bib{Grundhoferbis}{article}{
author = {Grundh\"ofer, Theo},
title = {Automorphism groups of compact projective planes},
journal = {Geom. Dedicata},
year = {1986},
volume = {21},
pages = {291--298},
number = {3}
}

\bib{GKK}{article}{
author = {Grundh\"ofer, Theo},
author = {Knarr, Norbert},
author = {Kramer, Linus},
title = {Flag-homogeneous compact connected polygons},
journal = {Geom. Dedicata},
year = {1995},
volume = {55},
pages = {95--114},
number = {1}
}

\bib{GKK2}{article}{
author = {Grundh\"ofer, Theo},
author = {Knarr, Norbert},
author = {Kramer, Linus},
title = {Flag-homogeneous compact connected polygons II},
journal = {Geom. Dedicata},
year = {2000},
volume = {83},
pages = {1--29},
number = {1--3}
}

\bib{Grundhofer}{article}{
author = {Grundh\"ofer, Theo},
author = {Kramer, Linus},
author = {Van Maldeghem, Hendrik},
author = {Weiss, Richard M.},
title = {Compact totally disconnected Moufang buildings},
journal = {Tohoku Math. J. (2)},
year = {2012},
volume = {64},
pages = {333--360},
number = {3}
}

\bib{Knarr}{article}{
author = {Knarr, Norbert},
title = {Projectivities of generalized polygons},
journal = {Ars Combin.},
year = {1988},
volume = {25},
number = {B},
pages = {265--275},
}

\bib{Lowen}{incollection}{
author = {L\"owen, Rainer},
title = {Projectivities and the geometric structure of topological planes},
booktitle = {Geometry - von Staudt's point of view},
series = {NATO Adv. Study Inst. Ser. C: Math. Phys. Sci.},
volume = {70},
publisher = {Reidel},
editor = {Plaumann, Peter},
editor = {Strambach, Karl},
year = {1981},
place = {Dordrecht},
pages = {339--372}
}

\bib{Munkres}{book}{
author = {Munkres, James R.},
title = {Topology},
edition = {2},
publisher = {Prentice Hall},
year = {2000}
}

\bib{Ronan}{book}{
author = {Ronan, Mark},
title = {Lectures on buildings},
series = {Perspectives in Mathematics},
volume = {7},
publisher = {Academic Press, Inc.},
place = {Boston},
year = {1989}
}

\bib{Salzmann}{book}{
author = {Salzmann, Helmut},
author = {Betten, Dieter},
author = {Grund\"ofer, Theo},
author = {H\"ahl, Hermann},
author = {L\"owen, Rainer},
author = {Stroppel, Markus},
title = {Compact projective planes},
subtitle = {with an introduction to octonion geometry},
series = {de Gruyter Exp. Math.},
volume = {21},
publisher = {Walter de Gruyter \& Co.},
place = {New York},
year = {1995},
}

\bib{Titsarbres}{incollection}{
author = {Tits, Jacques},
title = {Sur le groupe des automorphismes d'un arbre},
booktitle = {Essays on topology and related topics (Mémoires dédiés à Georges de Rham)},
language = {French},
publisher = {Springer},
editor = {Haefliger, André},
editor = {Narasimhan, Raghavan},
year = {1970},
place = {New York},
pages = {188--211}
}

\bib{Tits}{book}{
author = {Tits, Jacques},
title = {Buildings of spherical type and finite BN-pairs},
publisher = {Springer-Verlag},
year = {1974},
volume = {386},
series = {Lecture Notes in Math.},
place = {Berlin}
}

\bib{VM}{article}{
author = {Van Maldeghem, Hendrik},
author = {Van Steen, Kristel},
title = {Characterizations by automorphism groups of some rank~3 buildings},
subtitle = {I. Some properties of half strongly-transitive triangle buildings},
journal = {Geom. Dedicata},
year = {1998},
volume = {73},
number = {2},
pages = {119--142}
}

\bib{Weiss}{book}{
author = {Weiss, Richard M.},
title = {The structure of affine buildings},
series = {Ann. of Math. Stud.},
volume = {168},
publisher = {Princeton University Press},
place = {Princeton},
year = {2009}
}

\end{biblist}
\end{bibdiv}
\end{document}